\documentclass[11pt,reqno]{amsart}
\usepackage{amsrefs}
\usepackage{amsfonts, amsmath, amssymb, amscd, amsthm, bm, cancel}
\usepackage{url}
\usepackage{graphicx}
\usepackage[
colorlinks,citecolor=magenta,linkcolor=blue,urlcolor=magenta]{hyperref}
\usepackage[margin=1in]{geometry}
\usepackage{enumerate}

\newtheorem{thm}{Theorem}[section]
  \newtheorem*{thmA}{Theorem A}
  \newtheorem*{thmB}{Theorem B}
  \newtheorem*{thmC}{Theorem C}
  \newtheorem*{thmD}{Theorem D}
 
  \newtheorem{conjecture}[thm]{Conjecture}
  
 \newtheorem{lem}[thm]{Lemma}
 \newtheorem{assump}[thm]{Assumption}
 \newtheorem{prop}[thm]{Proposition}
 \theoremstyle{definition}
 
  \newtheorem*{ack}{Acknowledgments}
 \theoremstyle{remark}
 \newtheorem{rem}[thm]{Remark}

 \numberwithin{equation}{section}
 \allowdisplaybreaks

\renewcommand{\(}{\left(}
\renewcommand{\)}{\right)}

\renewcommand{\-}{\overline}

\newcommand{\R}{\mathbb{R}}

\renewcommand{\H}{\mathbb{H}}

\renewcommand{\a}{\alpha}

\renewcommand{\d}{\delta}

\renewcommand{\k}{\kappa}
\renewcommand{\l}{\lambda}

\newcommand{\D}{\Delta}

\renewcommand{\t}{\theta}
\newcommand{\s}{\sigma}

\newcommand{\G}{\Gamma}

\renewcommand{\L}{\Lambda}

\newcommand{\ra}{\rightarrow}

\newcommand{\mrm}{\mathrm}

\newcommand{\divv}{\mrm{div}}

\begin{document}
\title[Geometric inequalities for static convex domains in hyperbolic space]{Geometric inequalities for static convex domains in hyperbolic space}

\author{Yingxiang Hu}
\address{School of Mathematics, Beihang University, Beijing 100191, P.R. China}
\email{\href{mailto:huyingxiang@buaa.edu.cn}{huyingxiang@buaa.edu.cn}}

\author{Haizhong Li}
\address{Department of Mathematical Sciences, Tsinghua University, Beijing 100084, P.R. China}
\email{\href{mailto:lihz@tsinghua.edu.cn}{lihz@tsinghua.edu.cn}}

\date{\today}


\begin{abstract}
We prove that the static convexity is preserved along two kinds of locally constrained curvature flows in hyperbolic space. Using the static convexity of the flow hypersurfaces, we prove new family of geometric inequalities for such hypersurfaces in hyperbolic space.
\end{abstract}

\maketitle
\tableofcontents

\section{Introduction}\label{sec:1}

For any bounded domain $\Omega$ with smooth boundary $M=\partial \Omega$ in hyperbolic space $\mathbb H^{n+1}$, the $k$th quermassintegral $W_k$ is defined as the measure of the set of totally geodesic $k$-dimensional subspaces which intersect $\Omega$. It can be expressed as a linear combination of integral of $k$th mean curvatures and of the enclosed volume (see \cite{Sant2004})
\begin{align*}
& W_0(\Omega)= ~\mathrm{Vol}(\Omega),\qquad W_{1}(\Omega)=~\frac 1{n+1}|M|,\\
& W_{k+1}(\Omega)=~ \frac 1{n+1} \int_M E_k d\mu-\frac{k}{n+2-k}W_{k-1}(\Omega),\quad k=1,\cdots,n,
\end{align*}
where $E_k$ is the normalized $k$-th mean curvature of $M$. Along any outward normal variation with speed $\mathcal{F}$, the $k$th quermassintegral $W_k$ evolves by (see e.g. \cite[(3.5)]{WX14})
\begin{align}\label{s2:variation-quermassintegral}
\frac{d}{dt}W_k(\Omega_t)=~\frac{n+1-k}{n+1}\int_{M_t} E_k \mathcal{F}d\mu_t, \quad 0\leq k\leq n.
\end{align}

In order to establish the quermassintegral inequalities in hyperbolic space, one natural choice is the following globally constrained quermassintegral preserving flow
\begin{align}\label{s1:QP-MCF}
\frac{\partial}{\partial t}X=\( \frac{\int_{M_t}E_k^\frac{1}{k-l}E_{l}^{1-\frac{1}{k-l}}d\mu_t}{\int_{M_t}E_{l}d\mu_t}-\(\frac{E_k}{E_{l}}\)^\frac{1}{k-l}\) \nu, \quad 0\leq l<k\leq n,
\end{align}   
Along this flow, the $l$th quermassintegral $W_{l}(\Omega_t)$ is preserved while the $k$th quermassintegral $W_k(\Omega_t)$ is decreasing. Therefore, the smooth convergence of the flow \eqref{s1:QP-MCF} to geodesic spheres would imply the quermassintegral inequalities. For $l=0$ and $k=1$, this flow is called the volume preserving mean curvature flow, which was first studied by Cabezas-Rivas and Miquel \cite{Cabez07}. By imposing {\em h-convexity}\footnote{A closed hypersurface in hyperbolic space is called {\em h-convex} if its principal curvatures satisfy $\k_i\geq 1$ for all $i$.} on the initial hypersurface, they proved that the flow converges smoothly to a geodesic sphere as $t\ra \infty$. Later, the smooth convergence of the flow \eqref{s1:QP-MCF} with $0\leq l<k\leq n$ was established by Wang and Xia \cite{WX14}, which yields the following quermassintegral inequalities in hyperbolic space.

\begin{thmA}\cite{WX14}
	Let $\Omega$ be a bounded and h-convex domain with smooth boundary in $\mathbb H^{n+1}$. Then there holds
	\begin{align}\label{WX-ineq}
	W_k(\Omega) \geq f_k \circ f_{l}^{-1}(W_{l}(\Omega)), \quad 0\leq l<k\leq n.	
	\end{align}
	Equality holds in \eqref{WX-ineq} if and only if $\Omega$ is a geodesic ball. Here $f_k:[0,\infty)\ra \mathbb R^{+}$ is a monotone function defined by $f_k(r)=W_k(B_r)$, the $k$th quermassintegral for the geodesic ball of radius $r$, and $f_l^{-1}$ is the inverse function of $f_l$.
\end{thmA}

The quermassintegral inequalities in Euclidean space have been proved by Guan and Li \cite{GL09} for $k$-convex\footnote{A smooth domain is called {\em $k$-convex} (resp. {\em strictly $k$-convex}) if the principal curvatures of its boundary satisfy $\k\in \overline{\G}_k^{+}$ (resp. $\k\in \G_{k}^{+}$), where $\G_k^{+}$ is the Garding cone defined in \S \ref{s2:sec-2.1}.} star-shaped domains. Thus, it is natural to generalize Theorem A to a larger class of smooth bounded domains in hyperbolic space. In this direction, the inequality \eqref{WX-ineq} with $k=3$, $l=1$ was proved earlier by the second named author with Wei and Xiong in \cite{LWX14} for star-shaped domains with strictly $2$-convex boundary. Ge, Wang and Wu \cite{GeWW14} proved the inequalities \eqref{WX-ineq} with $k=2m+1$ $(0<2m<n-1)$ and $l=1$ for h-convex domains. Later, the authors of this paper \cite{HL19} generalized Ge-Wang-Wu's inequalities to smooth bounded domains with {\em nonnegative sectional curvature}\footnote{A hypersurface in hyperbolic space has {\em nonnegative} (resp. {\em positive}) {\em sectional curvature} if its principal curvatures satisfy $\k_i\k_j\geq 1$ (resp. $\k_i\k_j>1$) for all distinct $i,j$.} boundary. Recently, the authors of this paper with Andrews \cite{AHL19} proved the inequalities \eqref{WX-ineq} with $k=n-1$ and $l=n-1-2m$ $(0<2m<n-1)$ for strictly convex domains. On the other hand, the proof of convergence of the globally constrained quermassintegral preserving flow \eqref{s1:QP-MCF} depends heavily on {\em h-convexity} of hypersurfaces, see e.g. \cite{AW18,Cabez07,WX14}. The convergence of the flow \eqref{s1:QP-MCF} with $k=1,\cdots, n$ and $l=0$ was recently proved by Andrews, Chen and Wei \cite{ACW2018} for smooth bounded domains with {\em positive sectional curvature} boundary, and henceforth the inequalities \eqref{WX-ineq} with $k=1,\cdots,n$ and $l=0$ hold for such domains. 

Alternative choice to prove the quermassintegral inequalities is the following locally constrained quermassintegral preserving flows. The hyperbolic space can be viewed as a warped product space $\mathbb H^{n+1}=[0,\infty)\times \mathbb S^{n}$ equipped with the metric $\-g=dr^2+\l^2(r)\s$, where $\l(r)=\sinh r$ is the warping factor and $\s$ is the round metric of the unit sphere $\mathbb{S}^n$. Brendle, Guan and Li \cite{BGL} introduced the following locally constrained inverse curvature flow in $\mathbb H^{n+1}$ (see also a recent survey \cite{GL19} by Guan and Li):
\begin{align}\label{s1:BGL-flow}
\frac{\partial}{\partial t}X=\(\l'(r)\frac{E_{k-1}}{E_k}-u\)\nu,\quad 1\leq k\leq n,
\end{align}
where $\l'(r)=\cosh r$ and $u=\langle \l\partial_r,\nu\rangle$ is the support function of the flow hypersurface $M_t=X(t,\cdot)$. In view of the Minkowski formula \eqref{s2:Minkowski-formula}, this flow preserves $W_k(\Omega_t)$ and increases $W_{k-1}(\Omega_t)$ simultaneously, provided that the flow hypersurface is strictly $k$-convex and star-shaped. It is challenging to establish the convergence of this flow \eqref{s1:BGL-flow} under the assumption that the initial hypersurface is strictly $k$-convex and star-shaped, which would prove the quermassintegral inequalities for such domains. In \cite{BGL}, Brendle, Guan and Li proved the convergence of the flow \eqref{s1:BGL-flow} with $k=n$ for strictly convex hypersurfaces, and hence the inequalities \eqref{WX-ineq} with $k=n$ and $l=0,1,\cdots,n-1$ hold for strictly convex domains. Recently, the authors of this paper with Wei \cite{Hu-Li-Wei2020} proved the convergence of this flow \eqref{s1:BGL-flow} with $k=1,\cdots,n$ for h-convex hypersurfaces in hyperbolic space, which yields a new proof of Theorem A.

Besides the quermassintegral inequalities, there is of great interest to establish the weighted geometric inequalities in hyperbolic space. In \cite{Scheuer-Xia2019}, Scheuer and Xia  introduced the following locally constrained inverse curvature flow in $\mathbb H^{n+1}$:
\begin{align}\label{s1:SX-ICF}
\frac{\partial}{\partial t}X=&\(\frac{1}{F}-\frac{u}{\l'(r)}\)\nu.
\end{align}
In particular, if $F=E_k/E_{k-1}$, $k=1,\cdots,n$, they proved that the following convergence result.
\begin{thmB}\cite{Scheuer-Xia2019}
	Let $X_0$ be a smooth embedding of a closed $n$-dimensional manifold $M$ in $\mathbb H^{n+1}$ such that $M_0=X_0(M)$ is star-shaped and strictly $k$-convex along $X_0(M)$. Then any solution $M_t=X(M,t)$ of \eqref{s1:SX-ICF} with $F=E_k/E_{k-1}$ exists for $t>0$ and it converges to a geodesic sphere centered at the origin in the $C^\infty$-topology as $t\ra \infty$. Moreover, the flow hypersurface $M_t=X(t,M)$ is star-shaped and strictly $k$-convex for $t>0$. 
\end{thmB}

These assumptions are quite similar to the purely inverse curvature flow, while the locally constrained term in \eqref{s1:SX-ICF} will be helpful in establishing geometric inequalities. In particular, employing the flow \eqref{s1:SX-ICF} with $F=E_1$, they proved the following Minkowski type inequality. 
\begin{thmC}\cite{Scheuer-Xia2019}
	Let $\Omega$ be a star-shaped and strictly mean convex domain with smooth boundary $M$ in $\mathbb H^{n+1}$. Then there holds
	\begin{align}\label{SX-ineq-I}
	\int_{M} \l' E_1 d\mu \geq (n+1)\int_{\Omega} \l' d\mrm{vol}+\omega_n^{\frac{2}{n+1}}\((n+1)\int_{\Omega}\l' d\mrm{vol}\)^{\frac{n-1}{n+1}},
	\end{align}
	where $\omega_n$ is the area of the unit sphere $\mathbb S^{n}\subset \mathbb R^{n+1}$. Equality holds in \eqref{SX-ineq-I} if and only if $\Omega$ is a geodesic ball centered at the origin.
\end{thmC}

For a bounded domain $\Omega$ in $\mathbb H^{n+1}$ with smooth boundary $M=\partial\Omega$, it is called {\em static convex} (resp. {\em strictly static convex}) if its second fundamental form satisfies 
\begin{align*}
h_{ij} \geq \frac{u}{\l'}g_{ij}>0, \quad (\text{resp. $h_{ij}>\frac{u}{\l'}g_{ij}>0$}) \quad \text{everywhere on $M$.}
\end{align*}
The static convexity implies the strict convexity, but it is weaker than h-convexity since $\frac{u}{\l'}<1$. Combining \eqref{SX-ineq-I} with the following Minkowski type inequality (see Xia \cite[Theorem 1.1]{Xia2016})
\begin{align*}
\(\int_{M} \l' d\mu \)^2 \geq (n+1) \int_M \l' E_1 d\mu \int_{\Omega} \l' d\mrm{vol}
\end{align*}
for hypersurfaces satisfying $h_{ij}\geq \frac{u}{\l'}g_{ij}$ everywhere, Scheuer and Xia \cite[Theorem 1.6]{Scheuer-Xia2019} proved the following weighted isoperimetric inequality.
\begin{thmD}\cite{Scheuer-Xia2019}
	Let $\Omega$ be a static convex domain with smooth boundary $M$ in $\mathbb H^{n+1}$. Then there holds
	\begin{align}\label{SX-ineq-III}
	\int_{M} \l' d\mu \geq \(\((n+1)\int_{\Omega}\l' d\mrm{vol}\)^{2}+\omega_n^{\frac{2}{n+1}}\((n+1)\int_{\Omega}\l'd\mrm{vol}\)^\frac{2n}{n+1}\)^\frac{1}{2}.
	\end{align}
	Equality holds in \eqref{SX-ineq-III} if and only if $\Omega$ is a geodesic ball centered at the origin.
\end{thmD}

Motivated by Theorems C and D, for any bounded domain $\Omega$ with smooth boundary $M=\partial \Omega$, we introduce the {\em weighted curvature integral} as follows: 
\begin{align*}
W_{0}^{\l'}(\Omega)=&~\int_{M} u d\mu=\int_{\Omega}(n+1)\l'd\mrm{vol}, \quad W_{n+1}^{\l'}(\Omega)=\int_{M}\l'E_n d\mu,\\
W_{k}^{\l'}(\Omega)=&~\int_{M}\l' E_{k-1}d\mu=~\int_{M} u E_k d\mu, \quad k=1,\cdots,n.
\end{align*}
Along the outward normal variation with speed $\mathcal{F}$, the weighted curvature integral evolves by (see Proposition \ref{s2:evol-weighted-curvature-integral})
\begin{align*}
\frac{d}{dt} W_{k}^{\l'}(\Omega_t)=&~\int_{M_t} \(k u E_{k-1}+(n+1-k) \l' E_k\)\mathcal{F} d\mu_t, \quad k=0,\cdots,n+1,
\end{align*}
where we take $E_{-1}=E_{n+1}=0$ by convention.

We would like to propose the following conjectures for the weighted curvature integrals in hyperbolic space.
\begin{conjecture}\label{conjecture}
	Let $\Omega$ be a static convex domain with smooth boundary in $\mathbb H^{n+1}$. Then there holds
	\begin{align}\label{weighted-quermassintegral-ineq}
	W_{k}^{\l'}(\Omega)\geq h_k \circ h_{l}^{-1}(W_{l}^{\l'}(\Omega)), \quad 0\leq  l<k \leq n+1.
	\end{align}
	Equality holds in \eqref{weighted-quermassintegral-ineq} if and only if $\Omega$ is a geodesic ball centered at the origin. Here $h_k:[0,\infty)\ra \mathbb R^{+}$ is a monotone function defined by $h_k(r)=W_k^{\l'}(B_r)=\omega_n \sinh^{n+1-k} r \cosh^{k} r$, the $k$th weighted curvature integral for a geodesic ball of radius $r$, and $h_l^{-1}$ is the inverse function of $h_l$.
\end{conjecture}
The stronger form of Conjecture \ref{conjecture} is as follows.
\begin{conjecture}\label{conjecture-strong-version}
	Let $0\leq l<k\leq n+1$. Let $\Omega$ be a star-shaped and $(k-1)$-convex domain with smooth boundary in $\mathbb H^{n+1}$. Then there holds
	\begin{align*}
	W_{k}^{\l'}(\Omega)\geq h_k \circ h_{l}^{-1}(W_{l}^{\l'}(\Omega)).
	\end{align*}
	Equality holds if and only if $\Omega$ is a geodesic ball centered at the origin.
\end{conjecture}

\subsection{Main results}
In this paper, we first introduce a new locally constrained flow. Let $X_0:M^n \ra \mathbb H^{n+1}$ be a smooth embedding such that $M_0$ is a closed, star-shaped hypersurface in $\mathbb H^{n+1}$. We consider the smooth family of immersions $X:M^n \times [0,T)\ra \mathbb H^{n+1}$ satisfying the following evolution equations:
\begin{align}\label{s1:locally-MCF}
\frac{\partial}{\partial t}X(x,t)=\(1-\frac{u F}{\l'(r)} \)\nu(x,t),
\end{align}
where $F=E_1$.

Our new observation in this paper is that the static convexity is preserved along a large class of locally constrained curvature flows (Theorem \ref{thm-static-convexity}) including the flows \eqref{s1:SX-ICF} and \eqref{s1:locally-MCF}, provided that $F$ satisfies the following
\begin{assump}\label{s1:Assumption}
	\begin{enumerate}[(i)]
		\item $F(\mathcal{W})=f(\k(\mathcal{W}))$, where $\k$ is the eigenvalues of $\mathcal{W}$ and $f$ is a smooth symmetric function on the positive cone $\G_{+}=\{(x_i)\in\mathbb R^n ~:~ x_i>0\}$ satisfying 
		\begin{enumerate}[(1)]
			\item $f$ is strictly increasing, i.e., $\dot{f}^i=\partial f/\partial \k_i>0$ on $\G_{+}$, $\forall i=1,\cdots,n$;
			\item $f$ is homogeneous of degree $1$, i.e., $f(k\k)=kf(\k)$ for any $k>0$;
			\item $f$ is strictly positive on $\G_{+}$ and is normalized such that $f(1,\cdots,1)=1$;			
		\end{enumerate}
		\item $f$ is concave.
		\item $f$ is inverse concave, i.e., the function
		\begin{align}\label{s1:inverse-concave}
		f_{\ast}(x_1,\cdots,x_n)=f(x_1^{-1},\cdots,x_n^{-1})^{-1}
		\end{align}
		is concave.
	\end{enumerate}
\end{assump}  
\begin{rem}
	Important examples of the curvature function $F$ satisfying Assumption \ref{s1:Assumption} include the curvature quotients $F=(E_{k}/E_{l})^{1/(k-l)}$, $0\leq l<k\leq n$, see \cite{And07}.
\end{rem}

We first prove the convergence of the flow \eqref{s1:locally-MCF} with $F=E_1$ for star-shaped hypersurfaces. This is a weighted volume preserving flow, which is inspired by Guan-Li's mean curvature type flow \cite{GL15}.
\begin{thm}\label{s1:main-thm-I}
	Let $X_0$ be a smooth embedding of a closed $n$-dimensional manifold $M$ in $\mathbb H^{n+1}$ such that $M_0=X_0(M)$ is star-shaped. Then any solution $M_t=X(M,t)$ of \eqref{s1:locally-MCF} with $F=E_1$ remains star-shaped for $t>0$ and it converges to a geodesic sphere $\partial B_{r_\infty}$ centered at the origin in the $C^\infty$-topology as $t\ra \infty$, where the radius $r_\infty$ is uniquely determined by $W^{\l'}_0(B_{r_\infty})=W^{\l'}_0(\Omega_0)$. Moreover, if the initial hypersurface $M_0=X_0(M)$ is static convex, then the flow hypersurface $M_t=X(t,M)$ becomes strictly static convex for $t>0$. 
\end{thm}

We also prove that the static convexity is preserved along the flow \eqref{s1:SX-ICF} for $t>0$.
\begin{thm}\label{s1:main-thm-III}
	Assume that $F$ satisfies Assumption \ref{s1:Assumption}. Let $X_0$ be a smooth embedding of a closed $n$-dimensional manifold $M$ in $\mathbb H^{n+1}$ such that $X_0(M)$ is static convex. Then the flow hypersurface $M_t=X(t,M)$ of \eqref{s1:SX-ICF} becomes strictly static convex for $t>0$. 
\end{thm}
 
As applications, we obtain the following geometric inequalities between weighted curvature integrals and quermassintegrals. We emphasize that the preservance of static convexity along the flows will be crucial in the proof of the geometric inequalities.

Firstly, we apply both the flow \eqref{s1:locally-MCF} with $F=E_1$ and the flow \eqref{s1:SX-ICF} with $F=E_{k-1}/E_{k-2}$ to prove the Conjecture \ref{conjecture} with $l=0$ and $1\leq k\leq n+1$ for static convex domains.
\begin{thm}\label{thm-weighted-quermassintegral-ineq-I}
	Let $\Omega$ be a static convex domain with smooth boundary $M$ in $\mathbb H^{n+1}$. For $1\leq k\leq n+1$, there holds
	\begin{align}\label{weighted-quermassintegral-ineq-I}
	W_{k}^{\l'}(\Omega) \geq h_k \circ h_0^{-1}(W_0^{\l'}(\Omega)).
	\end{align}
	Equivalently,
	\begin{align}\label{s1:explicit-form}
	\int_{M} \l' E_{k-1} d\mu\geq \omega_n \(\(\frac{(n+1)\int_{\Omega}\l' d\mrm{vol}}{\omega_n}\)^\frac{2}{k}+\(\frac{(n+1)\int_{\Omega}\l' d\mrm{vol}}{\omega_n}\)^\frac{2(n-k+1)}{(n+1)k}\)^\frac{k}{2}.
	\end{align}
	Equality holds in \eqref{weighted-quermassintegral-ineq-I} if and only if $\Omega$ is a geodesic ball centered at the origin. 
\end{thm}
\begin{rem}
	The inequality \eqref{s1:explicit-form} with $k=1$ is Theorem D. By Theorem C, the inequality \eqref{s1:explicit-form} with $k=2$ also holds for star-shaped domains with strictly mean convex boundary.
\end{rem}

Using the flow \eqref{s1:locally-MCF} with $F=E_1$, we prove the weighted geometric inequality between the weighted enclosed volume and the volume for star-shaped domains.
\begin{thm}\label{thm-geometric-ineq}
	Let $\Omega$ be a smooth star-shaped domain in $\mathbb H^{n+1}$. Then there holds
	\begin{align}\label{geom-ineq-1}
	W_{0}^{\l'}(\Omega)=(n+1)\int_{\Omega}\l' d\mrm{vol}\geq h_{0} \circ f_0^{-1}(W_0(\Omega)).
	\end{align}
	Equality holds in \eqref{geom-ineq-1} if and only if $\Omega$ is a geodesic ball centered at the origin. 
\end{thm}

Applying the the flow \eqref{s1:SX-ICF} with $F=E_k/E_{k-1}$, we proved the following weighted Alexandrov-Fenchel inequalities for static convex domains.
\begin{thm}\label{thm-geometric-ineq-I}
	Let $\Omega$ be a static convex domain with smooth boundary $M$ in $\mathbb H^{n+1}$. For $0 \leq k \leq n$ and $0\leq m\leq k$, there holds
	\begin{align}\label{geom-ineq-2}
	W_{k+1}^{\l'}(\Omega)=\int_{M} \l'E_{k} d\mu \geq h_{k+1} \circ f_m^{-1}(W_m(\Omega)).
	\end{align}
	Equality holds in \eqref{geom-ineq-2} if and only if $\Omega$ is a geodesic ball centered at the origin. 
\end{thm}
\begin{rem}
	\begin{enumerate}[(i)]
		\item The inequality \eqref{geom-ineq-2} with $k=1$ and $m=1$ was proved by de Lima and Girao \cite{deLima-Girao2016} for star-shaped domains with strictly mean convex boundary. For odd $k$ and $m=1$, the inequality \eqref{geom-ineq-2} was proved by Ge, Wang and Wu \cite{Ge-Wang-Wu2015} for h-convex domains.
		For $1\leq k\leq n$ and $0\leq m \leq k$, the inequality \eqref{geom-ineq-2} for h-convex domains was recently proved by the authors of the paper with Wei in \cite[Theorem 1.4]{Hu-Li-Wei2020}. For $k=n$ and $0\leq m \leq n$, by the work \cite{BGL,GL19}, the inequality \eqref{geom-ineq-2} also holds for strictly convex domains.
		
		\item In \cite{Ge-Wang-Wu2015}, Ge, Wang and Wu introduced the Gauss-Bonnet-Chern mass $m_l^{\mathbb H}$ for asymptotically hyperbolic graphs. Using the weighted Alexandrov-Fenchel inequality \eqref{geom-ineq-2} with $k=2l-1$ and $m=1$, they proved an optimal Penrose type inequality for $m_l^{\mathbb H}$, under the assumption that the boundary of each component is h-convex, see \cite[Theorem 1.6]{Ge-Wang-Wu2015}. By Theorem \ref{thm-geometric-ineq-I}, we can weaken this assumption therein to be static convex.
	\end{enumerate}
\end{rem}

We would also like to propose the following conjecture on weighted Alexandrov-Fenchel inequalities.
\begin{conjecture}\label{conjecture-geometric-inequality}
	Let $0 \leq k \leq n$ and $0\leq m\leq k$. Let $\Omega$ be a star-shaped and $k$-convex domain in $\mathbb H^{n+1}$ with smooth boundary $M$. Then there holds
	\begin{align*}
	W_{k+1}^{\l'}(\Omega)=\int_{M} \l'E_{k} d\mu \geq h_{k+1} \circ f_m^{-1}(W_m(\Omega)).
	\end{align*}
	Equality holds if and only if $\Omega$ is a geodesic ball centered at the origin.
\end{conjecture}

The paper is organized as follows. In \S \ref{sec:2}, we collect some basic properties of symmetric functions and geometry of hypersurfaces in hyperbolic space. In \S \ref{sec:3}, we derive the evolution equations along two kinds of locally constrained curvature flows. In \S \ref{sec:4}, we use the nonparametric form of these flows to derive the $C^0$-estimates of the radial function. In \S \ref{sec:5}, we apply the tensor maximum principle to show that the static convexity is preserved along the two kinds of locally constrained curvature flows, which includes the flows \eqref{s1:SX-ICF} and \eqref{s1:locally-MCF} as special cases. Theorem \ref{s1:main-thm-III} then follows. In \S \ref{sec:7}, we complete the proof of Theorem \ref{s1:main-thm-I}. In \S \ref{sec:8}, we apply these locally constrained curvature flows to prove geometric inequalities in hyperbolic space, including Theorems \ref{thm-weighted-quermassintegral-ineq-I}, \ref{thm-geometric-ineq} and \ref{thm-geometric-ineq-I}.

\begin{ack}
	This work was partially supported by NSFC grant No.11831005 and NSFC-FWO grant No.1196131001. The authors would like to thank Professor Yong Wei for helpful discussions.
\end{ack}

\section{Preliminaries}\label{sec:2} 
\subsection{Properties of symmetric functions}$\ $\label{s2:sec-2.1}
We first review some properties of symmetric functions on the positive cone $\G_{+}\subset \mathbb R^n$. Given a smooth symmetric function $F(A)=f(\k(A))$, where $A=(A_{ij}) \in \mrm{Sym}(n)$ is a symmetric matrix and $\k(A)=(\k_1,\cdots,\k_n)$ gives the eigenvalues of $A$. We assume that $f$ is a smooth, symmetric, positive, strictly increasing, $1$-homogeneous function on $\G_{+}$ and is normalized such that $f(1,\cdots,1)=1$. We denote by $\dot{F}^{ij}$ and $\ddot{F}^{ij,kl}$ the first and second derivative with respect to the components of its argument, so that
\begin{align*}
\left.\frac{\partial}{\partial s}F(A+sB) \right|_{s=0} = \dot{F}^{ij}(A)B_{ij}
\end{align*} 
and
\begin{align*}
\left.\frac{\partial^2}{\partial s^2} F(A+sB)\right|_{s=0} = \ddot{F}^{ij,kl}(A) B_{ij} B_{kl}
\end{align*}
for any two symmetric matrices $A$, $B$. We also use the notations $\dot{f}^i(\k)$, $\dot{f}^{ij}(\k)$ to denote the derivatives of $f$ with respect to $\k$. At any diagonal $A$, we have
\begin{align*}
\dot{F}^{ij}(A)=\dot{f}^i(\k(A))\d_i^j.
\end{align*}
If diagonal $A$ has distinct eigenvalues, the second derivatives $\ddot{F}$ of $F$ in direction $B\in\mrm{Sym}(n)$ is given in terms of $\ddot{f}$ and $\dot{f}$ by (see \cite{And94,And07}):
\begin{align}\label{s2:second-derivative-expr}
\ddot{F}^{ij,kl}B_{ij}B_{kl}=\sum_{i,k}\ddot{f}^{ik} B_{ii}B_{kk}+2\sum_{i>k}\frac{\dot{f}^i-\dot{f}^k}{\k_i-\k_k}B_{ik}^2.
\end{align}
This formula makes sense as a limit in the case of any repeated values of $\k_i$. 

For any positive definite symmetric matrix $A\in \operatorname{Sym}(n)$ with eigenvalues $\k(A)\in \G_{+}$, define $F_{\ast}(A)=F(A^{-1})^{-1}$. Then $F_\ast(A)=f_\ast(\k(A))$, where $f_\ast$ is the dual function of $f$ defined in \eqref{s1:inverse-concave}. We say that $F$ is inverse concave if $F_{\ast}(A)$ is concave. Since $f$ is defined on the positive cone $\G_{+}$, the following lemma charaterizes the concavity and the inverse concavity of $f$ and $F$.
\begin{lem}(\cite{And07,AMZ13})
	\begin{enumerate}[(i)]
		\item $f$ is concave if and only if the following matrix 
		\begin{align}\label{f-concave}
		(\ddot{f}^{kl})\leq 0.
		\end{align}	
		\item $F$ is concave if and only if $f$ is concave and	
		\begin{align}\label{f-concave-II}
		\frac{\dot{f}^k-\dot{f}^l}{\k_k-\k_l} \leq 0,\quad \forall k\neq l. 
		\end{align}	
		\item $f$ is inverse concave if and only if the following matrix
		\begin{align}\label{f-inverse-concave}
		\(\ddot{f}^{kl}+2\frac{\dot{f}^k}{\k_k}\d_{kl}\) \geq 0.
		\end{align}
		\item $F$ is inverse concave if and only if $f$ is inverse concave and
		\begin{align}\label{f-inverse-concave-II}
		\frac{\dot{f}^k-\dot{f}^l}{\k_k-\k_l}+\frac{\dot{f}^k}{\k_l}+\frac{\dot{f}^l}{\k_k} \geq 0,~~\forall k\neq l. 
		\end{align}
	\end{enumerate}
\end{lem}
\begin{rem}
	Since $f$ and $f_{\ast}$ are defined on $\G_{+}$, we have (see \cite[Corollaries 5.3 \& 5.5]{And07})
	\begin{enumerate}[(i)]
		\item $F$ is concave if and only if $f$ is concave;
		\item $F$ is inverse concave if and only if $f$ is inverse concave.
	\end{enumerate} 
\end{rem}
Important examples of concave and inverse functions include $(E_k/E_l)^{1/(k-l)}$ with $k>l$, where
$$
E_k(\k)=\binom{n}{k}^{-1}\sum_{1\leq i_1<\cdots<i_k \leq n}\k_{i_1}\cdots \k_{i_k}.
$$ 
is the normalized $k$-th elementary symmetric function. It is convenient to set $E_0(\kappa)=1$ and $E_k(\kappa)=0$ for $k>n$. For any symmetric matrix $A=(A_{ij})\in \operatorname{Sym}(n)$, we set $E_k(A)=E_k(\k(A))$, then $E_k(A)$ is defined by
$$
E_k(A)=\frac{(n-k)!}{n!}\d_{i_1 \cdots i_k}^{j_1\cdots j_k} A_{i_1j_1} \cdots A_{i_kj_k}, \quad k=1,\cdots,n,
$$
where $\d_{i_1 \cdots i_k}^{j_1\cdots j_k}$ is generalized Kronecker delta defined by 
\begin{align*}
\d^{j_1j_2\cdots j_k}_{i_1i_2\cdots i_k}=\det\(\begin{matrix}
\d^{j_1}_{i_1} & \d^{j_1}_{i_2} & \cdots & \d^{j_1}_{i_k} \\
\d^{j_2}_{i_1} & \d^{j_2}_{i_2} & \cdots & \d^{j_2}_{i_k} \\
\vdots        &    \vdots      & \vdots & \vdots \\
\d^{j_k}_{i_1} & \d^{j_k}_{i_2} & \cdots & \d^{j_k}_{i_k}
\end{matrix}\).
\end{align*}
The Garding cone is defined by 
$$
\G_{k}^{+}=\{\k \in \mathbb R^n ~|~ E_m(\k)>0, \forall m\leq k \}.
$$ 
Then $n$-convex is convex in usual sense, $1$-convex is referred as mean convex. We also take $\G_{0}^{+}=\mathbb R^{n}$ by convention.
The following Newton-MacLaurin inequalities are well-known, see e.g. \cite[Lemma 2.5 on p.55]{Guan12}.
\begin{lem}\label{lem-newton-maclaurin-ineq}
	Let $1\leq k \leq n$. For $\k \in \G_{k}^{+}$, we have
	\begin{align}\label{s2:newton-maclaurin-ineq}
	E_{k+1}(\k) E_{k-1}(\k) \leq E_k^2(\k), \quad E_{k+1}(\k) \leq E_k^{\frac{k+1}{k}}(\k).
	\end{align}
	Equality holds in \eqref{s2:newton-maclaurin-ineq} if and only if $\k_1=\cdots=\k_n$.
\end{lem}

Note that $E_k(\k)=E_k(A)$, where $\k$ gives the eigenvalues of $A$. Denote by $\dot{E}_k^{i}=\frac{\partial E_k}{\partial \k_i}(\k)$ and  
$\dot{E}_k^{ij}(A)=\frac{\partial E_k(A)}{\partial A_{ij}}$.
\begin{lem}\label{lem-Newton-tensor} We have
	\begin{align}
	\sum_{i,j}\dot{E}_k^{ij}A_{ij}=&\sum_{i}\dot{E}_k^{i}\k_i= k E_k,\label{newton-formula-1}\\
	\sum_{i,j}\dot{E}_k^{ij}\d_{i}^{j} =&\sum_{i}\dot{E}_k^{i}= k E_{k-1},\label{newton-formula-2}\\
	\sum_{i,j}\dot{E}_k^{ij}(A^2)_{ij}=&\sum_{i}\dot{E}_k^{i}\k_i^2= n E_1 E_{k} -(n-k)E_{k+1},\label{newton-formula-3}
	\end{align}
	where $(A^2)_{ij}=\sum_{l}A_{il}A_{lj}$.
\end{lem}

\subsection{Hypersurfaces in hyperbolic space}$\ $ \label{s2-2}
The hyperbolic space $\mathbb H^{n+1}$ can be expressed as a warped product manifold $\mathbb R^+\times \mathbb S^n$ equipped with the metric
$$
\-g=dr^2+\l(r)^2 \s,
$$
where $\l(r)=\sinh r$ and $\s$ is the round metric of the unit sphere $\mathbb S^n \subset \mathbb R^{n+1}$. We define 
$$
\Lambda(r)=\int_0^{r} \l(s)ds=\l'(r)-1,
$$
where $\l'(r)=\cosh r$. Let $\-\nabla$ be the Levi-Civita connection with respect to $\-g$. The vector field $\-\nabla \Lambda=\l\partial_r$ on $\mathbb H^{n+1}$ is a conformal Killing field, i.e.,
\begin{align}\label{s2:conformal-Killing}
\-\nabla(\l \partial_r)=\l' \-g.
\end{align}

Let $M$ be a closed smooth hypersurface in $\mathbb H^{n+1}$ with unit outward normal $\nu$. The second fundamental form $h$ of $M$ is given by $h(X,Y)=\langle \-\nabla \nu, Y\rangle$ for any tangent vectors $X,Y$ on $M$. The principal curvatures $\k=(\k_1,\cdots,\k_n)$ are the eigenvalues of the second fundamental form $h$ with respect to the induced metric $g$ on $M$. In a local coordinate $(x^1,\cdots,x^n)$ of $M$, we denote $g_{ij}=g(\partial_i,\partial_j)$ and $h_{ij}=h(\partial_i,\partial_j)$. Then the Weingarten matrix is given by $\mathcal{W}=(h_i^j)=(g^{jk}h_{ki})$, where $(g^{ij})$ is the inverse matrix of $(g_{ij})$. Then the principal curvatures $\k$ of $M$ are the eigenvalues of the Weingarten matrix $\mathcal{W}$.

We recall the following lemmas for smooth hypersurfaces in $\mathbb H^{n+1}$, see e.g. \cite[Lemmas 2.2 \& 2.6]{GL15}.
\begin{lem}\label{lem-1} Let $(M,g)$ be a smooth hypersurface in $\mathbb H^{n+1}$. Then $\Lambda|_{M}$ satisfies
	\begin{align}\label{s2:2.1}
	\nabla_i\l'=\nabla_i \Lambda= \langle \l\partial_r,e_i\rangle,\quad \nabla_j\nabla_i \l'=\nabla_j\nabla_i \Lambda=\l' g_{ij} -u h_{ij},
	\end{align}
	and the support function $u=\langle \l\partial_r,\nu \rangle$ satisfies
	\begin{align}\label{s2:2.2}
	\nabla_i u=  \langle \l\partial_r,e_k\rangle h_i^k, \quad \nabla_j\nabla_i u= \langle \l\partial_r,\nabla h_{ij}\rangle+\l' h_{ij}-u (h^2)_{ij},
	\end{align}
	where $\{e_1,\cdots,e_n\}$ is a basis of the tangent space of $M$.
\end{lem}

By the divergence-free property of $\dot{E}_k^{ij}(h)$ and \eqref{s2:2.1}, we have the well-known Minkowski formulas (see e.g. \cite[Proposition 2.5]{GL15}).
\begin{lem}\label{lem-Minkowski-formula} Let $(M,g)$ be a smooth closed hypersurface in $\mathbb H^{n+1}$. Then there holds
	\begin{align}\label{s2:Minkowski-formula}
	\int_M \l' E_{k-1} d\mu= \int_{M} u E_k d\mu,  \quad k=1,\cdots,n.
	\end{align}
\end{lem}

\subsection{Parametrization by Radial graph}\label{s2:sec-2.2}
A smooth closed hypersurface $M$ in hyperbolic space $\mathbb H^{n+1}$ is called {\em star-shaped} if its support function $u= \langle \l\partial_r,\nu\rangle>0$ everywhere on $M$. This is equivalent to that the hypersurface $M$ can be expressed as a radial graph in spherical coordinates $(r(\t),\t)$ in $\mathbb H^{n+1}$, that is,
\begin{align*}
M= \{(r(\t),\t)\in \mathbb R^{+} \times \mathbb S^n ~|~\t \in \mathbb S^n \}.
\end{align*}
Let $\t=(\t^1,\cdots,\t^n)$ be a local coordinate of the round sphere $(\mathbb S^n,\s)$. Let $D$ be the Levi-Civita connection on $(\mathbb S^n,\s)$. Let $\partial_i=\partial_{\t^i}$ and $r_i=D_i r$. For the convenience of the notations, we define $\varphi:\mathbb S^n \ra \mathbb R$ by $\varphi(\t)=\Psi(r(\t))$, where $\Psi=\Psi(r)$ is a positive smooth function satisfies $\Psi'(r)=1/\l(r)$. Let $\varphi_i=D_i \varphi$ and $\varphi_{ij}=D_iD_j\varphi$. Then the tangential vector on $M$ are given by $\{X_i= \partial_i +r_i \partial_r=\partial_i+\l \varphi_i\partial_r, i=1,\cdots,n\}$.
The induced metric $g_{ij}$ of $M=\operatorname{graph}r$ can be expressed as
\begin{align}\label{s2:induced-metric}
g_{ij} = \l^2 \s_{ij}+r_i r_i=\l^2 (\s_{ij}+\varphi_i\varphi_j),
\end{align}
where $\s_{ij}=\s(\partial_i,\partial_j)$. Denote by $v=\sqrt{1+\l^{-2}|Dr|^2}=\sqrt{1+|D\varphi|^2}$, the inverse matrix $(g^{ij})$ of $(g_{ij})$ is given by
\begin{align}\label{s2:inverse-metric}
g^{ij}=\frac{1}{\l^2}\(\s^{ij}-\frac{\varphi^i \varphi^j}{v^2}\),
\end{align}
where $\varphi^i=\s^{ij}\varphi_j$. The unit outward normal is given by
\begin{align*}
\nu=\frac{1}{v}\(\partial_r - \frac{\varphi^i}{\l}\partial_i\).
\end{align*}
It follows that the support function $u=\langle \l\partial_r,\nu\rangle=\frac{\l}{v}$. The second fundamental form $h_{ij}$ and the Weingarten matrix $h_i^j$ of $M$ can be expressed as (see e.g., \cite{Ge11})
\begin{align*}
h_{ij}= \frac{\l'}{\l v}g_{ij}-\frac{\l}{v}\varphi_{ij},
\end{align*}
and 
\begin{align}\label{s2:2nd-fundamental-form}
h_i^j=g^{jk}h_{ki}= \frac{\l'}{\l v}\d_i^j-\frac{1}{\l v}\(\s^{jk}-\frac{\varphi^j\varphi^k}{v^2}\)\varphi_{ki}. 
\end{align}

\section{Evolution equations}$\ $ \label{sec:3} 
Along the general flow 
\begin{align}\label{s3:evol-general-flow}
\frac{\partial}{\partial t}X=\mathcal{F}\nu,
\end{align}
in hyperbolic space $\mathbb H^{n+1}$, we have the following evolution equations for the induced metric $g_{ij}$, the unit outward normal $\nu$, the area element $d\mu_t$, the second fundamental form $h_{ij}$ and the Weingarten matrix $h_i^j$ of the flow hypersurfaces $M_t=X(M^n,t)$: (see e.g., \cite{LWX14})
	\begin{align}
	\frac{\partial}{\partial t}g_{ij}=&2\mathcal{F}h_{ij}, \label{s3:3.1} \\
	\frac{\partial}{\partial t}\nu=&-\nabla \mathcal{F},  \label{s3:3.2} \\
	\frac{d}{dt}d\mu_t=& nE_1 \mathcal{F} d\mu_t, \label{s3:3.3} \\
	\frac{\partial}{\partial t}h_{ij}=&-\nabla_j\nabla_i \mathcal{F}+\mathcal{F}((h^2)_{ij}+g_{ij}), \label{s3:3.4}\\
	\frac{\partial}{\partial t}h_i^j=&-\nabla^j\nabla_i\mathcal{F}-\mathcal{F}((h^2)_i^j-\d_i^j), \label{s3:3.5}
	\end{align}
where $\nabla$ denotes the Levi-Civita connection with respect to the induced metric $g_{ij}$ on $M_t$. 

First of all, we deduce the variational formulas for the weighted curvature integrals along the general flow \eqref{s3:evol-general-flow} in hyperbolic space.
\begin{prop}\label{s2:evol-weighted-curvature-integral}
	Let $M_t=\partial\Omega_t \subset \mathbb H^{n+1}$ be a smooth family of closed hypersurfaces evolves by the flow \eqref{s3:evol-general-flow}. Then there holds
	\begin{align}\label{s6:evol-weighted-curvature-integral}
	\frac{d}{dt}W_{k}^{\l'}(\Omega_t)=\int_{M_t} \( k u E_{k-1}+(n+1-k)\l'E_k\)\mathcal{F}d\mu_t, \quad 0\leq k\leq n+1.
	\end{align}
\end{prop}
\begin{proof}
	For $k=0$, it is well-known. For $1\leq k\leq n+1$, a proof can be found in \cite[\S 5.2]{Hu-Li-Wei2020}. We include it here for
	convenience. Along the flow \eqref{s3:evol-general-flow}, we have
	\begin{align*}
	\frac{\partial}{\partial t}\l' = \langle \-\nabla \l', \partial_t\rangle =u\mathcal{F},
	\end{align*}
	and
	\begin{align*}
	\frac{\partial}{\partial t}E_{k-1} =\frac{\partial E_{k-1}}{\partial h_i^j} \partial_t h_i^j =\dot{E}_{k-1}^{ij}\(-\nabla_i\nabla_j \mathcal{F}-\mathcal{F}((h^2)_{ij}-g_{ij})\),
	\end{align*}
	where we used \eqref{s3:3.5}. Combining these with \eqref{s3:3.3}, we get
	\begin{align*}
	\frac{d}{dt}W^{\l'}_k(\Omega_t) =&\int_{M_t} (u E_{k-1}+n\l' E_{k-1} E_1) \mathcal{F}d\mu_t+\int_{M_t} \l'\dot{E}_{k-1}^{ij}\(-\nabla_i\nabla_j \mathcal{F}-\mathcal{F}((h^2)_{ij}-g_{ij})\)d\mu_t.
	\end{align*}
	Since $\dot{E}_k^{ij}$ is divergence-free, by integration by parts we obtain
	\begin{align*}
	\frac{d}{dt}W^{\l'}_k(\Omega_t)=&\int_{M_t} \(u E_{k-1}+n\l' E_{k-1}E_1-\l'\dot{E}_{k-1}^{ij}((h^2)_{ij}-g_{ij})\)\mathcal{F}d\mu_t -\int_{M_t}(\dot{E}_{k-1}^{ij}\nabla_i\nabla_j \l')\mathcal{F}d\mu_t \\
	=&\int_{M_t} \(u E_{k-1}+n\l' E_{k-1}E_1-\l'\dot{E}_{k-1}^{ij}(h^2)_{ij}+u\dot{E}_{k-1}^{ij}h_{ij}\)\mathcal{F}d\mu_t \\
	=&\int_{M_t} \(ku E_{k-1}+(n+1-k)\l'E_k\) \mathcal{F}d\mu_t,
	\end{align*}
	where we used \eqref{s2:2.1} in the second equality and \eqref{newton-formula-1}, \eqref{newton-formula-3} in the last equality.	
\end{proof}

Let $\Phi \in C^\infty(\mathbb R_{+})$ be a smooth function satisfying $\Phi'(s)=\frac{d}{ds}\Phi(s)>0$ for all $s>0$. We assume that $F$ satisfies (i) in Assumption \ref{s1:Assumption}. We consider the following two kinds of flows 
\begin{align}\label{s2:general-form-flow-I}
\frac{\partial}{\partial t}X= \(\Phi(1)-\Phi(\frac{u}{\l'}F)\)\nu,
\end{align}
and
\begin{align}\label{s2:general-form-flow-II}
\frac{\partial}{\partial t}X= \( \Phi(\frac{\l'}{u})-\Phi(F)\)\nu.
\end{align}
It is easy to see that 
\begin{enumerate}[(i)]
	\item The flow \eqref{s1:locally-MCF} is corresponding to the flow \eqref{s2:general-form-flow-I} with $\Phi(s)=s$;
	\item The flow \eqref{s1:SX-ICF} is corresponding to the flow \eqref{s2:general-form-flow-II} with $\Phi(s)=-s^{-1}$.
\end{enumerate}
Now we deduce the evolution equations along the flows \eqref{s2:general-form-flow-I} and \eqref{s2:general-form-flow-II}, respectively. For simplicity, we denote by 
\begin{align}\label{s2:notation-I}
\dot{\Phi}^{kl}(\frac{u}{\l'}F)=\Phi'(\frac{u}{\l'}F)\frac{u}{\l'}\dot{F}^{kl},\quad \ddot{\Phi}^{kl,pq}(\frac{u}{\l'}F)=\Phi'(\frac{u}{\l'}F)\frac{u}{\l'}\ddot{F}^{kl,pq}+\Phi''(\frac{u}{\l'}F)\frac{u^2}{\l'^2}\dot{F}^{kl}\dot{F}^{pq},
\end{align}
and
\begin{align}\label{s2:notation-II}
\dot{\Phi}^{kl}(F)=\Phi'(F)\dot{F}^{kl},\quad \ddot{\Phi}^{kl,pq}(F)=\Phi'(F)\ddot{F}^{kl,pq}+\Phi''(F)\dot{F}^{kl}\dot{F}^{pq}.
\end{align}
\begin{lem}\label{s3:lem-evolution-eq-1} Along the flow \eqref{s2:general-form-flow-I}, i.e., $\mathcal{F}=\Phi(1)-\Phi(\frac{u}{\l'}F)$, we have the following evolution equations.
	\begin{enumerate}[(i)]
		\item 
		\begin{align}\label{evol-static-function-I}
		\frac{\partial}{\partial t}(\frac{u}{\l'})=&\dot{\Phi}^{kl}(\frac{u}{\l'}F) \nabla_k \nabla_l (\frac{u}{\l'})+\frac{2}{\l'}\dot{\Phi}^{kl}(\frac{u}{\l'}F)\nabla_k(\frac{u}{\l'})\nabla_l \l'+\Phi'(\frac{u}{\l'}F)\frac{F}{\l'}\langle \l\partial_r,\nabla (\frac{u}{\l'})\rangle \nonumber\\
		&+(1-\frac{u^2}{\l'^2})(\Phi(1)-\Phi(\frac{u}{\l'}F))-(1+\frac{u^2}{\l'^2})\Phi'(\frac{u}{\l'}F)\frac{u}{\l'}F \nonumber\\
		&+\Phi'(\frac{u}{\l'}F)\frac{u^2}{\l'^2}\dot{F}^{kl}((h^2)_{kl}+g_{kl}).
		\end{align}
		\item 
		\begin{align}\label{evol-second-fundamental-form-I}
		\frac{\partial}{\partial t}h_{ij}
		=&\dot{\Phi}^{kl}(\frac{u}{\l'}F)\nabla_k \nabla_l h_{ij}+\ddot{\Phi}^{kl,pq}(\frac{u}{\l'}F)\nabla_i h_{kl}\nabla_j h_{pq} +\Phi'(\frac{u}{\l'}F)\frac{F}{\l'}\langle \l\partial_r,\nabla h_{ij}\rangle\nonumber\\
		& -2\Phi'(\frac{u}{\l'}F)\(\frac{F}{\l'}\nabla_{(i}\l'\nabla_{j)}(\frac{u}{\l'})-\nabla_{(i}(\frac{u}{\l'})\nabla_{j)}F\) \nonumber\\
		&+\Phi''(\frac{u}{\l'}F)\( F^2\nabla_i(\frac{u}{\l'})\nabla_j(\frac{u}{\l'})+2\frac{u}{\l'}F\nabla_{(i}(\frac{u}{\l'})\nabla_{j)}F\) \nonumber\\
		&+\Phi'(\frac{u}{\l'}F)\(\frac{u}{\l'}\dot{F}^{kl}((h^2)_{kl}+g_{kl})+F(1+\frac{u^2}{\l'^2}) \)h_{ij} \nonumber\\
		&+\(\Phi(1)-\Phi(\frac{u}{\l'}F)-2\Phi'(\frac{u}{\l'}F)\frac{u}{\l'}F\)((h^2)_{ij}+g_{ij}),
		\end{align}
		where the brackets denote symmetrization, e.g., $\nabla_{(i}\l'\nabla_{j)}(\frac{u}{\l'})=\frac{1}{2}(\nabla_{i}\l'\nabla_{j}(\frac{u}{\l'})+\nabla_{j}\l'\nabla_{i}(\frac{u}{\l'}))$.
		\item Taking $S_{ij}=h_{ij}-\frac{u}{\l'}g_{ij}$, we have
		\begin{align}\label{evol-Sij-I}
		\frac{\partial}{\partial t}S_{ij}=&\dot{\Phi}^{kl}(\frac{u}{\l'}F)\nabla_k \nabla_l S_{ij}+ \ddot{\Phi}^{kl,pq}(\frac{u}{\l'}F)\nabla_i h_{kl}\nabla_j h_{pq} +\Phi(\frac{u}{\l'}F)\frac{F}{\l'}\langle \l\partial_r,\nabla S_{ij}\rangle \nonumber\\
		&-2\Phi'(\frac{u}{\l'}F)\(\frac{F}{\l'}\nabla_{(i}\l'\nabla_{j)}(\frac{u}{\l'})-\nabla_{(i}(\frac{u}{\l'})\nabla_{j)}F\) \nonumber \\ &+\Phi''(\frac{u}{\l'}F)\( F^2\nabla_i(\frac{u}{\l'})\nabla_j(\frac{u}{\l'})+2\frac{u}{\l'}F\nabla_{(i}(\frac{u}{\l'})\nabla_{j)}F\) \nonumber\\
		&-\dot{\Phi}^{kl}(\frac{u}{\l'}F)\frac{2}{\l'}\nabla_k(\frac{u}{\l'})\nabla_l\l' g_{ij}+\(\Phi(1)-\Phi(\frac{u}{\l'}F)-2\Phi'(\frac{u}{\l'}F)\frac{u}{\l'}F\)((S^2)_{ij}+2\frac{u}{\l'}S_{ij})\nonumber\\
		&+\(\Phi'(\frac{u}{\l'}F)\(\frac{u}{\l'}\dot{F}^{kl}((h^2)_{kl}+g_{kl})+F(1+\frac{u^2}{\l'^2})\)-2\frac{u}{\l'}(\Phi(1)-\Phi(\frac{u}{\l'}F))\)S_{ij}.
		\end{align}
	\end{enumerate}
\end{lem}
\begin{proof}
	\begin{enumerate}[(i)]
		\item Since $u=\langle \l\partial_r,\nu\rangle$ and $\-\nabla\l'=\-\nabla \L=\l\partial_r$, we get
		\begin{align*}
		\frac{\partial}{\partial t}\l'=\langle \-\nabla \l', \partial_t X\rangle=\langle \l\partial_r,(\Phi(1)-\Phi(\frac{u}{\l'}F)\nu\rangle = u\(\Phi(1)-\Phi(\frac{u}{\l'}F)\).
		\end{align*}
		By \eqref{s2:2.1} and $1$-homogeneity of $F$, we have
		\begin{align*}
		\dot{\Phi}^{kl}(\frac{u}{\l'}F)\nabla_k \nabla_l \l'=\Phi'(\frac{u}{\l'}F)\frac{u}{\l'}\dot{F}^{kl}(\l'g_{kl}-uh_{kl})
		                                                    =\Phi'(\frac{u}{\l'}F)u\(\dot{F}^{kl}g_{kl}-\frac{u}{\l'}F\).
		\end{align*}
		Then we deduce that
		\begin{align}\label{evol-cosh-r}
		\frac{\partial}{\partial t}\l'=\dot{\Phi}^{kl}\nabla_k \nabla_l \l'+u\(\Phi(1)-\Phi(\frac{u}{\l'}F)-\Phi'(\frac{u}{\l'}F)\dot{F}^{kl}g_{kl}+\Phi'(\frac{u}{\l'}F)\frac{u}{\l'}F\).
		\end{align}
		On the other hand, by the conformal property \eqref{s2:conformal-Killing} of the vector field $\l\partial_r$ and \eqref{s3:3.2}, we have
		\begin{align*}
		\frac{\partial}{\partial t}u=&\langle \-\nabla_{\partial_t}(\l\partial_r),\nu\rangle+\langle \l\partial_r, \frac{\partial}{\partial t}\nu\rangle \\
		=&\l'\(\Phi(1)-\Phi(\frac{u}{\l'}F)\)+\langle \l\partial_r, \nabla \Phi(\frac{u}{\l'}F) \rangle \\
		=&\l'\(\Phi(1)-\Phi(\frac{u}{\l'}F)\)+\Phi'(\frac{u}{\l'}F) \(\frac{u}{\l'}\langle \l\partial_r, \nabla F\rangle+F\langle \l\partial_r, \nabla(\frac{u}{\l'})\rangle\).
		\end{align*}
		By \eqref{s2:2.2} and $1$-homogeneity of $F$, we also have
		\begin{align*}
		\dot{\Phi}^{kl}(\frac{u}{\l'}F)\nabla_k \nabla_l u=&\Phi'(\frac{u}{\l'}F)\frac{u}{\l'}\dot{F}^{kl}\(\langle \l\partial_r,\nabla h_{kl}\rangle+\l' h_{kl}-u(h^2)_{kl}\) \\
		=&\Phi'(\frac{u}{\l'}F)\frac{u}{\l'}\(\langle \l\partial_r,\nabla F\rangle+\l' F-u\dot{F}^{kl}(h^2)_{kl}\).
		\end{align*}
		Then we get
		\begin{align}\label{evol-support-function}
		\frac{\partial}{\partial t}u=&\dot{\Phi}^{kl}(\frac{u}{\l'}F)\nabla_k \nabla_l u+\Phi'(\frac{u}{\l'}F)F\langle \l\partial_r, \nabla(\frac{u}{\l'})\rangle \nonumber\\
		                             &+\l'\(\Phi(1)-\Phi(\frac{u}{\l'}F)-\Phi'(\frac{u}{\l'}F)\frac{u}{\l'}F+\Phi'(\frac{u}{\l'}F)\frac{u^2}{\l'^2}\dot{F}^{kl}(h^2)_{kl}\).
		\end{align}
		Combining \eqref{evol-cosh-r} with \eqref{evol-support-function}, we obtain
		\begin{align*}
		\frac{\partial}{\partial t}(\frac{u}{\l'})=&\frac{1}{\l'}\frac{\partial}{\partial t}u-\frac{u}{\l'^2}\frac{\partial}{\partial t}\l' \\
		=&\dot{\Phi}^{kl}(\frac{u}{\l'}F)\(\frac{1}{\l'}\nabla_k \nabla_l u-\frac{u}{\l'^2}\nabla_k \nabla_l \l'\)+\Phi'(\frac{u}{\l'}F)\frac{F}{\l'}\langle \l\partial_r,\nabla (\frac{u}{\l'})\rangle\\
		 &+(1-\frac{u^2}{\l'^2})(\Phi(1)-\Phi(\frac{u}{\l'}F))-(1+\frac{u^2}{\l'^2})\Phi'(\frac{u}{\l'}F)\frac{u}{\l'}F\\
		 &+\Phi'(\frac{u}{\l'}F)\frac{u^2}{\l'^2}\dot{F}^{kl}((h^2)_{kl}+g_{kl}).
		\end{align*}
		We also have
		\begin{align}\label{hessian-static-function}
		\nabla_k \nabla_l (\frac{u}{\l'})=&\frac{1}{\l'}\nabla_k \nabla_l u-\frac{u}{\l'^2}\nabla_k \nabla_l \l'-\frac{2}{\l'}\nabla_{(k}(\frac{u}{\l'})\nabla_{l)} \l'.
		\end{align}
		Thus, we get
		\begin{align*}
		\frac{\partial}{\partial t}(\frac{u}{\l'})=&\dot{\Phi}^{kl}(\frac{u}{\l'}F) \nabla_k \nabla_l (\frac{u}{\l'})+\dot{\Phi}^{kl}(\frac{u}{\l'}F)\frac{2}{\l'}\nabla_k(\frac{u}{\l'})\nabla_l \l'+\Phi'(\frac{u}{\l'}F)\frac{F}{\l'}\langle \l\partial_r,\nabla (\frac{u}{\l'})\rangle\\
		 &+(1-\frac{u^2}{\l'^2})(\Phi(1)-\Phi(\frac{u}{\l'}F))-(1+\frac{u^2}{\l'^2})\Phi'(\frac{u}{\l'}F)\frac{u}{\l'}F\\
		&+\Phi'(\frac{u}{\l'}F)\frac{u^2}{\l'^2}\dot{F}^{kl}((h^2)_{kl}+g_{kl}).
		\end{align*}
		
		\item By \eqref{s3:3.4}, we have
		\begin{align}\label{s3:evol-hij-1}
		\frac{\partial}{\partial t}h_{ij}=&-\nabla_j\nabla_i (\Phi(1)-\Phi(\frac{u}{\l'}F))+(\Phi(1)-\Phi(\frac{u}{\l'}F))((h^2)_{ij}+g_{ij})\nonumber \\
		=&\Phi'(\frac{u}{\l'}F)\(F\nabla_j\nabla_i (\frac{u}{\l'})+2\nabla_{(i}(\frac{u}{\l'})\nabla_{j)}F+\frac{u}{\l'}\nabla_j\nabla_i F\)\nonumber\\
		  &+\Phi''(\frac{u}{\l'}F)\( F^2\nabla_i(\frac{u}{\l'})\nabla_j(\frac{u}{\l'})+\frac{u^2}{\l'^2}\nabla_i F\nabla_j F+2\frac{u}{\l'}F\nabla_{(i}(\frac{u}{\l'})\nabla_{j)}F\)\nonumber\\
		  &+(\Phi(1)-\Phi(\frac{u}{\l'}F))((h^2)_{ij}+g_{ij}).
		\end{align}
		Combining the Gauss and Codazzi equations in hyperbolic space, we obtain the following generalized Simons' identity (see e.g., \cite{And94}):
		\begin{align*}
		\nabla_{(i}\nabla_{j)}h_{kl}=\nabla_{(k}\nabla_{l)}h_{ij}+((h^2)_{kl}+g_{kl})h_{ij}-h_{kl}((h^2)_{ij}+g_{ij}),
		\end{align*}
		where the brackets denote symmetrization. This yields
		\begin{align}\label{s3:Sims}
		\nabla_j\nabla_i F=& \nabla_j (\dot{F}^{kl}\nabla_i h_{kl}) \nonumber\\
		                  =&\dot{F}^{kl}\nabla_j\nabla_i h_{kl}+\ddot{F}^{kl,pq}\nabla_i h_{kl} \nabla_j h_{pq}\nonumber\\
		=&\dot{F}^{kl}\nabla_k \nabla_l h_{ij} + \dot{F}^{kl}((h^2)_{kl}+g_{kl})h_{ij}-F((h^2)_{ij}+g_{ij})+\ddot{F}^{kl,pq}\nabla_i h_{kl}\nabla_j h_{pq},
		\end{align}
		where we used the $1$-homogeneity of $F$. On the other hand, substituting \eqref{s2:2.1} and \eqref{s2:2.2} into \eqref{hessian-static-function}, we deduce that
		\begin{align}\label{s3:hessian-static-function}
		\nabla_j\nabla_i(\frac{u}{\l'})=&\frac{1}{\l'}\langle \l\partial_r,\nabla h_{ij}\rangle+(1+\frac{u^2}{\l'^2})h_{ij}-\frac{u}{\l'}((h^2)_{ij}+g_{ij})-\frac{2}{\l'}\nabla_{(i} \l' \nabla_{j)} (\frac{u}{\l'}).
		\end{align}
		Putting \eqref{s3:Sims} and \eqref{s3:hessian-static-function} into \eqref{s3:evol-hij-1}, in view of \eqref{s2:notation-I}, we obtain 
		\begin{align*}
		\frac{\partial}{\partial t}h_{ij}
		=&\dot{\Phi}^{kl}(\frac{u}{\l'}F)\nabla_k \nabla_l h_{ij}+\ddot{\Phi}^{kl,pq}(\frac{u}{\l'}F)\nabla_i h_{kl}\nabla_j h_{pq} +\Phi'(\frac{u}{\l'}F)\frac{F}{\l'}\langle \l\partial_r,\nabla h_{ij}\rangle\nonumber\\
		 & -2\Phi'(\frac{u}{\l'}F)\(\frac{F}{\l'}\nabla_{(i}\l'\nabla_{j)}(\frac{u}{\l'})-\nabla_{(i}(\frac{u}{\l'})\nabla_{j)}F\) \nonumber\\
		&+\Phi''(\frac{u}{\l'}F)\( F^2\nabla_i(\frac{u}{\l'})\nabla_j(\frac{u}{\l'})+2\frac{u}{\l'}F\nabla_{(i}(\frac{u}{\l'})\nabla_{j)}F\) \nonumber\\
		&+\Phi'(\frac{u}{\l'}F)\(\frac{u}{\l'}\dot{F}^{kl}((h^2)_{kl}+g_{kl})+F(1+\frac{u^2}{\l'^2}) \)h_{ij} \nonumber\\
		&+\(\Phi(1)-\Phi(\frac{u}{\l'}F)-2\Phi'(\frac{u}{\l'}F)\frac{u}{\l'}F\)((h^2)_{ij}+g_{ij}).
		\end{align*}
		
		\item Combining \eqref{s3:3.1}, \eqref{evol-static-function-I} with \eqref{evol-second-fundamental-form-I}, we get
		\begin{align*}
	\frac{\partial}{\partial t}S_{ij}=&\dot{\Phi}^{kl}(\frac{u}{\l'}F)\nabla_k \nabla_l S_{ij}+ \ddot{\Phi}^{kl,pq}(\frac{u}{\l'}F)\nabla_i h_{kl}\nabla_j h_{pq} +\Phi(\frac{u}{\l'}F)\frac{F}{\l'}\langle \l\partial_r,\nabla S_{ij}\rangle \nonumber\\
	&-2\Phi'(\frac{u}{\l'}F)\(\frac{F}{\l'}\nabla_{(i}\l'\nabla_{j)}(\frac{u}{\l'})-\nabla_{(i}(\frac{u}{\l'})\nabla_{j)}F\) \nonumber \\ &+\Phi''(\frac{u}{\l'}F)\( F^2\nabla_i(\frac{u}{\l'})\nabla_j(\frac{u}{\l'})+2\frac{u}{\l'}F\nabla_{(i}(\frac{u}{\l'})\nabla_{j)}F\) \nonumber\\
	&-\dot{\Phi}^{kl}(\frac{u}{\l'}F)\frac{2}{\l'}\nabla_k(\frac{u}{\l'})\nabla_l\l' g_{ij}+\(\Phi(1)-\Phi(\frac{u}{\l'}F)-2\Phi'(\frac{u}{\l'}F)\frac{u}{\l'}F\)((S^2)_{ij}+2\frac{u}{\l'}S_{ij})\nonumber\\
	&+\(\Phi'(\frac{u}{\l'}F)\(\frac{u}{\l'}\dot{F}^{kl}((h^2)_{kl}+g_{kl})+F(1+\frac{u^2}{\l'^2})\)-2\frac{u}{\l'}(\Phi(1)-\Phi(\frac{u}{\l'}F))\)S_{ij},
		\end{align*}
		where we also used $h_{ij}=S_{ij}+\frac{u}{\l'}g_{ij}$ and $(h^2)_{ij}=(S^2)_{ij}+2\frac{u}{\l'}S_{ij}+\frac{u^2}{\l'^2}g_{ij}$.
	\end{enumerate}
\end{proof}

\begin{lem}
Along the flow \eqref{s2:general-form-flow-II}, i.e., $\mathcal{F}=\Phi(\frac{\l'}{u})-\Phi(F)$, we have the following evolution equations.
\begin{enumerate}[(i)]
	\item 
	\begin{align}\label{evol-static-function-II}
		\frac{\partial}{\partial t}(\frac{u}{\l'})=&\dot{\Phi}^{kl}(F)\nabla_k\nabla_l (\frac{u}{\l'})+\dot{\Phi}^{kl}(F)\frac{2}{\l'}\nabla_k(\frac{u}{\l'})\nabla_l \l'+\Phi'(\frac{\l'}{u})\frac{\l'}{u^2}\langle \l\partial_r,\nabla (\frac{u}{\l'})\rangle \nonumber \\
	&+(1-\frac{u^2}{\l'^2})(\Phi(\frac{\l'}{u})-\Phi(F))-(1+\frac{u^2}{\l'^2})\Phi'(F)F+\frac{u}{\l'}\Phi'(F)\dot{F}^{kl}((h^2)_{kl}+g_{kl}).
	\end{align}
	\item 
	\begin{align}\label{evol-second-fundamental-form-II}
	\frac{\partial}{\partial t}h_{ij}=&\dot{\Phi}^{kl}(F)\nabla_k\nabla_l h_{ij}+\ddot{\Phi}^{kl,pq}(F)\nabla_i h_{kl}\nabla_j h_{pq}+\Phi'(\frac{\l'}{u})\frac{\l'}{u^2}\langle \l\partial_r,\nabla h_{ij}\rangle \nonumber\\
	&+\Phi'(\frac{\l'}{u})\frac{2}{u}\nabla_{(j}(\frac{\l'}{u})\nabla_{i)}u-\Phi''(\frac{\l'}{u})\nabla_i(\frac{\l'}{u})\nabla_j(\frac{\l'}{u}) \nonumber\\
	&+\(\Phi'(F)\dot{F}^{kl}((h^2)_{kl}+g_{kl})+(1+\frac{\l'^2}{u^2})\Phi'(\frac{\l'}{u})\)h_{ij} \nonumber\\
	&+\(\Phi(\frac{\l'}{u})-\Phi(F)-\Phi'(F)F-\Phi'(\frac{\l'}{u})\frac{\l'}{u}\)((h^2)_{ij}+g_{ij}),
	\end{align}
	\item Taking $S_{ij}=h_{ij}-\frac{u}{\l'}g_{ij}$, we have
	\begin{align}\label{evol-Sij-II}
	 \frac{\partial}{\partial t}S_{ij}=&\dot{\Phi}^{kl}(F)\nabla_k\nabla_l S_{ij}+\ddot{\Phi}^{kl,pq}(F)\nabla_i h_{kl}\nabla_j h_{pq}+\Phi'(\frac{\l'}{u})\frac{\l'}{u^2}\langle \l\partial_r,\nabla S_{ij}\rangle   \nonumber\\
	 &+\Phi'(\frac{\l'}{u})\frac{2}{u}\nabla_{(j}(\frac{\l'}{u})\nabla_{i)}u-\Phi''(\frac{\l'}{u})\nabla_i(\frac{\l'}{u})\nabla_j(\frac{\l'}{u})-\dot{\Phi}^{kl}(F)\frac{2}{\l'}\nabla_{k}(\frac{u}{\l'})\nabla_l \l' g_{ij}  \nonumber\\
	 &+\(\Phi(\frac{\l'}{u})-\Phi(F)-\Phi'(F)F-\Phi'(\frac{\l'}{u})\frac{\l'}{u} \)((S^2)_{ij}+2\frac{u}{\l'}S_{ij}) \nonumber\\
	 &+\(\Phi'(F)\dot{F}^{kl}((h^2)_{kl}+g_{kl})+(1+\frac{\l'^2}{u^2})\Phi'(\frac{\l'}{u})-2\frac{u}{\l'}(\Phi(\frac{\l'}{u})-\Phi(F))\) S_{ij}.
	\end{align}
\end{enumerate}
\end{lem}
\begin{proof}
	\begin{enumerate}[(i)]
		\item We have
		\begin{align*}
		\frac{\partial}{\partial t}\l'=\dot{\Phi}^{kl}(F)\nabla_k\nabla_l \l'+u\( \Phi(\frac{\l'}{u})-\Phi(F)-\frac{\l'}{u}\Phi'(F)\dot{F}^{kl}g_{kl}+\Phi'(F)F \),
		\end{align*}
		and
		\begin{align*}
		\frac{\partial}{\partial t}u=&\dot{\Phi}^{kl}(F)\nabla_k\nabla_l u-\Phi'(\frac{\l'}{u})\langle \l\partial_r,\nabla (\frac{\l'}{u})\rangle \\
		&+\l'\(\Phi(\frac{\l'}{u})-\Phi(F)-\Phi'(F)F+\frac{u}{\l'}\Phi'(F)\dot{F}^{kl}(h^2)_{kl}\).
		\end{align*}
		Combining these two equations with \eqref{hessian-static-function}, we derive
		\begin{align*}
		\frac{\partial}{\partial t}(\frac{u}{\l'})=&\dot{\Phi}^{kl}(F)\nabla_k\nabla_l (\frac{u}{\l'})+\dot{\Phi}^{kl}(F)\frac{2}{\l'}\nabla_k(\frac{u}{\l'})\nabla_l \l'+\Phi'(\frac{\l'}{u})\frac{\l'}{u^2}\langle \l\partial_r,\nabla (\frac{u}{\l'})\rangle \\
		&+(1-\frac{u^2}{\l'^2})(\Phi(\frac{\l'}{u})-\Phi(F))-(1+\frac{u^2}{\l'^2})\Phi'(F)F+\frac{u}{\l'}\Phi'(F)\dot{F}^{kl}((h^2)_{kl}+g_{kl}).
		\end{align*}
	    \item By \eqref{s3:3.4}, we have
	    \begin{align}\label{s3:evol-hij-2}
	    \frac{\partial}{\partial t}h_{ij}=&-\nabla_j\nabla_i(\Phi(\frac{\l'}{u})-\Phi(F))+(\Phi(\frac{\l'}{u})-\Phi(F))((h^2)_{ij}+g_{ij})\nonumber\\
	    =&\Phi'(F)\nabla_j\nabla_i F+\Phi''(F)\nabla_j F\nabla_i F-\Phi'(\frac{\l'}{u})\nabla_j\nabla_i(\frac{\l'}{u})-\Phi''(\frac{\l'}{u})\nabla_j(\frac{\l'}{u})\nabla_i(\frac{\l'}{u})\nonumber\\
	     &+(\Phi(\frac{\l'}{u})-\Phi(F))((h^2)_{ij}+g_{ij}).
	    \end{align}
	    We also have 
	    \begin{align}\label{s3:hessian-static-function-II}
	    \nabla_j \nabla_i (\frac{\l'}{u})=&\frac{1}{u}\nabla_j \nabla_i \l'-\frac{\l'}{u^2}\nabla_j \nabla_i u-\frac{2}{u}\nabla_{(j}(\frac{\l'}{u})\nabla_{i)} u\nonumber\\
	    =&\frac{\l'}{u}((h^2)_{ij}+g_{ij})-(1+\frac{\l'^2}{u^2})h_{ij}-\frac{\l'}{u^2}\langle \l\partial_r,\nabla h_{ij}\rangle-\frac{2}{u}\nabla_{(j}(\frac{\l'}{u})\nabla_{i)} u.
	    \end{align}
	    Substituting \eqref{s3:Sims} and \eqref{s3:hessian-static-function-II} into \eqref{s3:evol-hij-2}, in view of \eqref{s2:notation-II}, we derive
	    \begin{align*}
	     \frac{\partial}{\partial t}h_{ij}=&\dot{\Phi}^{kl}(F)\nabla_k\nabla_l h_{ij}+\ddot{\Phi}^{kl,pq}(F)\nabla_i h_{kl}\nabla_j h_{pq}+\Phi'(\frac{\l'}{u})\frac{\l'}{u^2}\langle \l\partial_r,\nabla h_{ij}\rangle \\
	     &+\Phi'(\frac{\l'}{u})\frac{2}{u}\nabla_{(j}(\frac{\l'}{u})\nabla_{i)}u-\Phi''(\frac{\l'}{u})\nabla_i(\frac{\l'}{u})\nabla_j(\frac{\l'}{u}) \\
	    &+\(\Phi'(F)\dot{F}^{kl}((h^2)_{kl}+g_{kl})+(1+\frac{\l'^2}{u^2})\Phi'(\frac{\l'}{u})\)h_{ij} \\
	    &+\(\Phi(\frac{\l'}{u})-\Phi(F)-\Phi'(F)F-\Phi'(\frac{\l'}{u})\frac{\l'}{u}\)((h^2)_{ij}+g_{ij}).
	    \end{align*}
	    
	    \item Using \eqref{s3:3.1}, \eqref{evol-static-function-II} and \eqref{evol-second-fundamental-form-II}, we obtain
	    \begin{align*}
	    \frac{\partial}{\partial t}S_{ij}=&\dot{\Phi}^{kl}(F)\nabla_k\nabla_l S_{ij}+\ddot{\Phi}^{kl,pq}(F)\nabla_i h_{kl}\nabla_j h_{pq}+\Phi'(\frac{\l'}{u})\frac{\l'}{u^2}\langle \l\partial_r,\nabla S_{ij}\rangle  \\
	    &+\Phi'(\frac{\l'}{u})\frac{2}{u}\nabla_{(j}(\frac{\l'}{u})\nabla_{i)}u-\Phi''(\frac{\l'}{u})\nabla_i(\frac{\l'}{u})\nabla_j(\frac{\l'}{u})-\dot{\Phi}^{kl}(F)\frac{2}{\l'}\nabla_{k}(\frac{u}{\l'})\nabla_l \l' g_{ij}\\
	    &+\(\Phi(\frac{\l'}{u})-\Phi(F)-\Phi'(F)F-\Phi'(\frac{\l'}{u})\frac{\l'}{u} \)((S^2)_{ij}+2\frac{u}{\l'}S_{ij})\\
	    &+\(\Phi'(F)\dot{F}^{kl}((h^2)_{kl}+g_{kl})+(1+\frac{\l'^2}{u^2})\Phi'(\frac{\l'}{u})-2\frac{u}{\l'}(\Phi(\frac{\l'}{u})-\Phi(F))\) S_{ij}.
	    \end{align*}
	    \end{enumerate}
\end{proof}

\section{$C^0$-estimates} \label{sec:4}
Along the general flow \eqref{s3:evol-general-flow}, if the evolving hypersurfaces $M_t, t\in [0,T^\ast)$ are star-shaped, we may parametrize each of them as a graph 
$$
M_t=\{(r(\t,t),\t)\in \mathbb R^+\times \mathbb S^n ~|~\t\in \mathbb S^n \},
$$ 
where $r(\cdot,t)$ is the radial function of $M_t$ defined on $\mathbb S^n$. Let $\varphi(\t,t)=\psi(r(\t,t))$. Then the flow \eqref{s3:evol-general-flow} can be reduced to a parabolic equation for $\varphi$. As long as the solution of \eqref{s3:evol-general-flow} exists and remains star-shaped, then $\varphi$ satisfies the following initial value problem (see \cite{GL15,Scheuer-Xia2019,Hu-Li-Wei2020}): 
\begin{align}\label{s2:evolution-parametric-form}
\left\{\begin{aligned}
\frac{\partial \varphi}{\partial t}=&\mathcal{F}\frac{v}{\l}, \quad (\t,t)\in \mathbb S^n \times [0,T^\ast), \\
\varphi(\cdot,0)=&\varphi_0(\cdot),\end{aligned} \right.
\end{align}
where $\varphi_0(\t)=\psi(r_0(\t))$ is determined by the radial function $r_0$ of the initial hypersurface $M_0$. The Eq. \eqref{s2:evolution-parametric-form} will be referred as the nonparametric form of the general flow \eqref{s3:evol-general-flow}.

Then the nonparametric form of the flow \eqref{s2:general-form-flow-I} or \eqref{s2:general-form-flow-II} is
	\begin{align}\label{s2:general-flow-I-parametric-form}
	\left\{\begin{aligned}
	\frac{\partial}{\partial t} \varphi=&\(\Phi(1)-\Phi(\frac{\l}{\l'v}F)\)\frac{v}{\l}, \quad (\t,t)\in \mathbb S^n \times [0,T^\ast), \\
	\varphi(\cdot,0)=&\varphi_0(\cdot),
	\end{aligned}\right.
	\end{align}
	or
	\begin{align}\label{s2:general-flow-II-parametric-form}
	\left\{\begin{aligned}
	\frac{\partial}{\partial t} \varphi=& \(\Phi(\frac{\l'v}{\l})-\Phi(F)\)\frac{v}{\l}, \quad (\t,t)\in \mathbb S^n \times [0,T^\ast), \\
	\varphi(\cdot,0)=&\varphi_0(\cdot).
	\end{aligned}\right.
	\end{align}
	
	We show that the solution $\varphi$ to \eqref{s2:general-flow-I-parametric-form} or \eqref{s2:general-flow-II-parametric-form} is uniformly bounded from above and below.
	\begin{prop}\label{s3:prop-C0-estimate} Let $\varphi_0$ be determined by the initial hypersurface $M_0$ in $\mathbb H^{n+1}$. If $\varphi=\varphi(\t,t), t\in [0,T^\ast)$ solves the initial value problem \eqref{s2:general-flow-I-parametric-form} or \eqref{s2:general-flow-II-parametric-form}, then
		\begin{align}\label{s2:C0-estimate}
		\min_{\t\in \mathbb S^n} \varphi_0(\t) \leq \varphi(\t,t) \leq \max_{\t\in \mathbb S^n} \varphi_0(\t), \quad \forall (\t,t)\in \mathbb S^n\times [0,T^\ast).
		\end{align}
	\end{prop}
	\begin{proof}
		We only prove the upper bound of $\varphi$ along the flow \eqref{s2:general-flow-I-parametric-form}, since the remaining proof is similar. At the maximum point of $\varphi$, we have
		\begin{align*}
		D \varphi=0, \quad v=1, \quad D^2\varphi \leq 0.
		\end{align*}
		In view of \eqref{s2:2nd-fundamental-form}, we get
		$$
		h_i^j=\frac{1}{\l}(-\varphi_{i}^{j}+\l'\d_{i}^{j}) \geq \frac{\l'}{\l}\d_{i}^{j}.
		$$
		Now we diagonalize the matrix $(\varphi_i^j)$, then $(h_i^j)=\mrm{diag}(\k_1,\cdots,\k_n)$ with $\k_i\geq \frac{\l'}{\l}$. This yields
		$$
		F(h_i^j)=f(\k_1,\cdots,\k_n) \geq f(\frac{\l'}{\l},\cdots,\frac{\l'}{\l})=\frac{\l'}{\l},
		$$
		where we used (i) in Assumption \ref{s1:Assumption}. As $\Phi'(s)>0$ for all $s>0$, at the maximum point of $\varphi$ we have
		\begin{align*}
		\frac{\partial}{\partial t}\varphi=\frac{1}{\l}\(\Phi(1)-\Phi(\frac{\l}{\l'}F)\)\leq 0.
		\end{align*}
		Then $\varphi(\t,t) \leq \max_{\t\in \mathbb S^n} \varphi_0(\t)$ follows from the maximum principle.
\end{proof}

\section{Preserving of static convexity}$\ $ \label{sec:5} 
In this section, we will use the tensor maximum principle to prove that the {\em static convexity} is preserved along the flow \eqref{s2:general-form-flow-I} or \eqref{s2:general-form-flow-II}, provided that
\begin{enumerate}[(i)]
	\item $\Phi$ satisfies $\Phi'(s)>0$ and $\Phi''(s)s+2\Phi'(s)\geq 0$ for all $s>0$;
	\item $F$ satisfies Assumption \ref{s1:Assumption}.
\end{enumerate}

\begin{rem}	
	(a) Important examples of the function $\Phi$ satisfying $\Phi'(s)>0$ and $\Phi''(s)s+2\Phi'(s)\geq 0$ for all $s>0$ include: (1) $\Phi(s)=s^{p}$ with $p>0$; (2) $\Phi(s)=-s^{-p}$ with $0<p\leq 1$; (3) $\Phi(s)=\ln s$.
	
    (b) Important examples of the curvature function $F$ satisfying Assumption \ref{s1:Assumption} include the curvature quotients $F=(E_{k}/E_{l})^{1/(k-l)}$, $0\leq l<k\leq n$. For more examples, we refer the readers to \cite[Section 2]{And07}.
\end{rem}

For convenience of the readers, we recall the tensor maximum principle, which was first proved by Hamilton \cite{Ham1982} and was generalized by Andrews \cite{And07}.
\begin{thm}[\cite{And07}]\label{thm-2}
	Let $S_{ij}$ be a smooth time-varying symmetric tensor field on a compact manifold $M$ satisfying
	\begin{align*}
	\frac{\partial}{\partial t}S_{ij} = a^{kl} \nabla_k \nabla_l S_{ij}+u^k \nabla_k S_{ij}+N_{ij},
	\end{align*}
	where $a^{kl}$ and $u$ are smooth, $\nabla$ is a (possibly time-dependent) smooth symmetric connection, and $a^{kl}$ is positive definite everywhere. Suppose that
	\begin{align}\label{3.2}
	N_{ij} v^i v^j + \sup_{\L} 2 a^{kl}(2\L_k^p \nabla_l S_{ip}v^i-\L_k^p \L_l^q S_{pq}) \geq 0,
	\end{align}
	whenever $S_{ij}\geq 0$ and $S_{ij}v^j=0$ and $\L$ is an $n\times n$-matrix. If $S_{ij}$ is positive definite everywhere on $M$ at $t=0$ and on $\partial M$ for $0\leq t\leq T$, then it is positive on $M\times [0,T]$.
\end{thm}

The main result of this section is the following.
\begin{thm}\label{thm-static-convexity}
	Assume that $F$ satisfies Assumption \ref{s1:Assumption} and $\Phi$ satisfies $\Phi'(s)>0$ and $\Phi''(s)s+2\Phi'(s)\geq 0$ for all $s>0$. If the initial hypersurface $M_0$ is a closed, static convex hypersurface in $\mathbb H^{n+1}$, then along the flow \eqref{s2:general-form-flow-I} or \eqref{s2:general-form-flow-II}, the evolving hypersurface $M_t$ becomes strictly static convex for $t>0$.
\end{thm}
\begin{proof}
    Along the flow \eqref{s2:general-form-flow-I} or \eqref{s2:general-form-flow-II}, the flow hypersurface $M_t$ is star-shaped at least for a short time $[0,T')\subset [0,T^\ast)$, where $T^{\ast}$ the maximal existence time of smooth solution of the flow \eqref{s2:general-form-flow-I} or \eqref{s2:general-form-flow-II}, respectively. Let $S_{ij}=h_{ij}-\frac{u}{\l'}g_{ij}$. On $[0,T')$, the static convexity of the flow hypersurface $M_t$ is equivalent to $S_{ij}\geq 0$, which implies that $M_t$ is strictly convex. At the minimum point of $u$ over $M_t$, by \eqref{s2:2.2} we have
    $$
    0=\nabla_i u=\k_i\nabla_i \l'=\k_i \l \nabla_i r, \quad 1\leq i\leq n.
    $$
    It follows from the strict convexity of $M_t$ that $Dr=0$ and $v=\sqrt{1+\l^{-2}|Dr|^2}=1$ at this point. By the expression of the support function $u=\frac{\l}{v}$, we get 
    $$
    \min_{M_t}u=\min_{\t\in \mathbb S^n}\l(r(\t,t)) \geq \min_{\t\in \mathbb S^n}\l(r_0(\t)),
    $$
    due to the $C^0$-estimates of $\varphi$ (Proposition \ref{s3:prop-C0-estimate}). This shows that the support function $u$ admits a uniform positive lower bound which is independent of time. Therefore, it suffices to show that $S_{ij}\geq 0$ on $[0,T')$, and a continuation argument implies that $T'=T^{\ast}$. 
    
    Along the flow \eqref{s2:general-form-flow-I}, by \eqref{evol-Sij-I} we have
	\begin{align}\label{s5:Q1}
	\frac{\partial}{\partial t}S_{ij}=&\dot{\Phi}^{kl}(\frac{u}{\l'}F)\nabla_k \nabla_l S_{ij}+ \ddot{\Phi}^{kl,pq}(\frac{u}{\l'}F)\nabla_i h_{kl}\nabla_j h_{pq} +\Phi(\frac{u}{\l'}F)\frac{F}{\l'}\langle \l\partial_r,\nabla S_{ij}\rangle \nonumber\\
	&-2\Phi'(\frac{u}{\l'}F)\(\frac{F}{\l'}\nabla_{(i}\l'\nabla_{j)}(\frac{u}{\l'})-\nabla_{(i}(\frac{u}{\l'})\nabla_{j)}F\) \nonumber \\ &+\Phi''(\frac{u}{\l'}F)\( F^2\nabla_i(\frac{u}{\l'})\nabla_j(\frac{u}{\l'})+2\frac{u}{\l'}F\nabla_{(i}(\frac{u}{\l'})\nabla_{j)}F\) \nonumber\\
	&-\dot{\Phi}^{kl}(\frac{u}{\l'}F)\frac{2}{\l'}\nabla_k(\frac{u}{\l'})\nabla_l\l' g_{ij}+\(\Phi(1)-\Phi(\frac{u}{\l'}F)-2\Phi'(\frac{u}{\l'}F)\frac{u}{\l'}F\)((S^2)_{ij}+2\frac{u}{\l'}S_{ij})\nonumber\\
	&+\(\Phi'(\frac{u}{\l'}F)\(\frac{u}{\l'}\dot{F}^{kl}((h^2)_{kl}+g_{kl})+F(1+\frac{u^2}{\l'^2})\)-2\frac{u}{\l'}(\Phi(1)-\Phi(\frac{u}{\l'}F))\)S_{ij},
	\end{align}
	while along the flow \eqref{s2:general-form-flow-II}, by \eqref{evol-Sij-II} we have
	\begin{align}\label{s5:Q2}
	\frac{\partial}{\partial t}S_{ij}=&\dot{\Phi}^{kl}(F)\nabla_k\nabla_l S_{ij}+\ddot{\Phi}^{kl,pq}(F)\nabla_i h_{kl}\nabla_j h_{pq}+\Phi'(\frac{\l'}{u})\frac{\l'}{u^2}\langle \l\partial_r,\nabla S_{ij}\rangle   \nonumber\\
	&+\Phi'(\frac{\l'}{u})\frac{2}{u}\nabla_{(j}(\frac{\l'}{u})\nabla_{i)}u-\Phi''(\frac{\l'}{u})\nabla_i(\frac{\l'}{u})\nabla_j(\frac{\l'}{u})-\dot{\Phi}^{kl}(F)\frac{2}{\l'}\nabla_{k}(\frac{u}{\l'})\nabla_l \l' g_{ij}  \nonumber\\
	&+\(\Phi(\frac{\l'}{u})-\Phi(F)-\Phi'(F)F-\Phi'(\frac{\l'}{u})\frac{\l'}{u} \)((S^2)_{ij}+2\frac{u}{\l'}S_{ij}) \nonumber\\
	&+\(\Phi'(F)\dot{F}^{kl}((h^2)_{kl}+g_{kl})+(1+\frac{\l'^2}{u^2})\Phi'(\frac{\l'}{u})-2\frac{u}{\l'}(\Phi(\frac{\l'}{u})-\Phi(F))\) S_{ij}.
	\end{align}
	
	To apply the tensor maximum principle, we need to show that \eqref{3.2} whenever $S_{ij}\geq 0$ and $S_{ij}v^j=0$ (so that $v$ is a null vector of $S$). Let $(x_0,t_0)$ be the point where $S_{ij}$ has a null vector $v$. We choose normal coordinates around $(x_0,t_0)$ such that $g_{ij}=\d_{ij}$ at this point. By continuity, we can assume that the principal curvatures are mutually distinct and in increasing order at $(x_0,t_0)$, that is $\k_1<\k_2<\cdots<\k_n$\footnote{This is possible since for any positive definite symmetric matrix $A$ with $A_{ij}\geq 0$ and $A_{ij}v^iv^j=0$ for some $v\neq 0$, there is a sequence of symmetric matrixes $\{A^{(k)}\}$ approaching $A$, satisfying $A^{(k)}_{ij}\geq 0$ and $A^{(k)}_{ij}v^iv^j=0$ and with each $A^{(k)}$ having distinct eigenvalues. Hence it suffices to prove the result in the case where all of $\kappa_i$ are distinct.}. The null vector condition $S_{ij}v^j=0$ implies that $v=e_1$ and $S_{11}=\k_1-\frac{u}{\l'}=0$ at $(x_0,t_0)$. The terms involving $S_{ij}$ and $(S^2)_{ij}$ satisfy the null vector condition and can be ignored. Moreover, by \eqref{s2:2.1}, \eqref{s2:2.2} and $\k_1=\frac{u}{\l'}$, we have 
	$$
	\nabla_1 (\frac{u}{\l'})=\frac{\l'\k_1-u}{\l'^2}\nabla_1 \l'=0, \quad \nabla_1 (\frac{\l'}{u})=\frac{u-\l'\k_1}{u^2}\nabla_1 \l'=0.
	$$
	Let $Q_1$ and $\hat{Q}_1$ be the remaining terms in the RHS of \eqref{s5:Q1} and \eqref{s5:Q2}, respectively. Since $\Phi'(s)>0$ for all $s>0$, it remains to show that
	\begin{align}\label{s4:Q1}
	\frac{Q_1}{\Phi'(\frac{u}{\l'}F)\frac{u}{\l'}}=&\ddot{F}^{kl,pq}\nabla_1 h_{kl}\nabla_1 h_{pq}+\frac{\Phi''(\frac{u}{\l'}F)}{\Phi'(\frac{u}{\l'}F)}\frac{u}{\l'}|\nabla_1F|^2-\frac{2}{\l'}\dot{F}^{kl}\nabla_k(\frac{u}{\l'})\nabla_l\l' \nonumber\\
	                                  &+2\sup_{\L} \dot{F}^{kl}(2\L_k^p \nabla_l S_{1p}-\L_k^p\L_l^q S_{pq}) \geq 0,
	\end{align}
	or 
	\begin{align}\label{s4:Q1-2}
	\frac{\hat{Q}_1}{\Phi'(F)}=&\ddot{F}^{kl,pq}\nabla_1 h_{kl}\nabla_1 h_{pq}+\frac{\Phi''(F)}{\Phi'(F)}|\nabla_1F|^2-\frac{2}{\l'}\dot{F}^{kl}\nabla_k(\frac{u}{\l'})\nabla_l\l' \nonumber\\
	                     &+2\sup_{\L} \dot{F}^{kl} (2\L_k^p \nabla_l S_{1p}-\L_k^p\L_l^q S_{pq}) \geq 0.
	\end{align}
	Since $\Phi''(s)s+2\Phi'(s)\geq 0$ for all $s>0$, we have
	$$
	\frac{\Phi''(\frac{u}{\l'}F)}{\Phi'(\frac{u}{\l'}F)}\frac{u}{\l'} \geq -\frac{2}{F}, \quad 	\frac{\Phi''(F)}{\Phi'(F)} \geq -\frac{2}{F}.
	$$
	On the other hand, by assumption, $S_{11}=0$ and $\nabla_k S_{11}=0$ at $(x_0,t_0)$, we have
	\begin{align}\label{s4:gradient-term}
	\dot{F}^{kl} (2\L_k^p \nabla_l S_{1p}-\L_k^p\L_l^q S_{pq})=&\sum_{k=1}^{n}\sum_{p=2}^{n}\dot{f}^k(2\L_k^p \nabla_k S_{1p}-(\L_k^p)^2 S_{pp}) \nonumber \\
	=& \sum_{k=1}^{n} \sum_{p=2}^{n}\dot{f}^k \(\frac{(\nabla_k S_{1p})^2}{S_{pp}}-\(\L_k^p-\frac{\nabla_k S_{1p}}{S_{pp}}\)^2 S_{pp}\).
	\end{align}
	Then the supremum of the last line in \eqref{s4:gradient-term} is obtained by choosing $\L_k^p=\frac{\nabla_k S_{1p}}{S_{pp}}$. Thus it suffices to check that
	\begin{align}\label{s4:Q1-I}
	\tilde{Q}_1=\ddot{F}^{kl,rs}\nabla_1 h_{kl}\nabla_1 h_{rs}-\frac{2}{F}|\nabla_1 F|^2-\frac{2}{\l'}\dot{F}^{kl}\nabla_k(\frac{u}{\l'})\nabla_l\l'+2\sum_{k=1}^{n}\sum_{l>1}\dot{f}^k \frac{(\nabla_k S_{1l})^2}{S_{ll}}\geq 0.
	\end{align}
	By Codazzi equations, we have 
	$$
	\nabla_k S_{1l}=\nabla_k h_{1l}-\nabla_k(\frac{u}{\l'})\d_{1l}=\nabla_1 h_{kl}, \quad \forall l>1.
	$$
	Since $0=\nabla_k S_{11}=\nabla_k (h_{11}-\frac{u}{\l'})$ for $k\geq 1$, we deduce that 
	\begin{align}\label{s4:nabla-h11}
	\nabla_k h_{11}=\nabla_k(\frac{u}{\l'})=\frac{\k_k-\frac{u}{\l'}}{\l'}\nabla_k\l', \quad \forall k\geq 1.
	\end{align}
	As in \cite[\S 3]{AW18}, we use \eqref{s2:second-derivative-expr} to express the second derivatives of $F$, and by \eqref{s4:nabla-h11} we compute that
	\begin{align}\label{s4:static-cov1}
	\tilde{Q}_1=&\ddot{f}^{kl}\nabla_1 h_{kk} \nabla_1 h_{ll}+2\sum_{k>l}\frac{\dot{f}^k-\dot{f}^l}{\k_k-\k_l}(\nabla_1 h_{kl})^2-\frac{2}{f}|\nabla_1 F|^2-2\sum_{l>1}\dot{f}^{l}\frac{\k_l-\frac{u}{\l'}}{\l'^2}|\nabla_l\l'|^2  \nonumber\\
	&+2\sum_{l>1}\dot{f}^1 \frac{\k_l-\frac{u}{\l'}}{\l'^2}|\nabla_l\l'|^2+2\sum_{k>1,l>1}\dot{f}^k \frac{(\nabla_1 h_{kl})^2}{\k_l-\frac{u}{\l'}}.
	\end{align}
	To make use of the inverse concavity of $f$, let $\tau_i=\k_i^{-1}$ and $f_{\ast}(\tau)=f(\k)^{-1}$. A direct calculation gives
	\begin{align*}
	\dot{f}^k=f_{\ast}^{-2}\frac{\partial f_\ast}{\partial \tau_k}\frac{1}{\k_k^2},\quad
	\ddot{f}^{kl}=-f_{\ast}^{-2} \frac{\partial^2 f_\ast}{\partial \tau_k\partial \tau_l}\frac{1}{\k_k^2\k_l^2}+2f^{-1}\dot{f}^{k}\dot{f}^{l}-2\frac{\dot{f}^k}{\k_k}\d_{kl}.
	\end{align*}
	By the concavity of $f_\ast$, the first term of \eqref{s4:static-cov1} can be estimated as follows
	\begin{align}\label{s4:inverse-concavity}
	\ddot{f}^{kl}\nabla_1 h_{kk} \nabla_1 h_{ll} \geq 2\(f^{-1}|\nabla_1 F|^2-\sum_{k}\frac{\dot{f}^k}{\k_k}(\nabla_1 h_{kk})^2\).
	\end{align}
	Substituting \eqref{s4:inverse-concavity} into \eqref{s4:static-cov1}, we get
	\begin{align}\label{s4:crucial-ineq}
	\tilde{Q}_1\geq &-2\sum_{k}\frac{\dot{f}^k}{\k_k}(\nabla_1 h_{kk})^2+2\sum_{k>l}\frac{\dot{f}^k-\dot{f}^l}{\k_k-\k_l}(\nabla_1 h_{kl})^2+2\sum_{k>1,l>1}\dot{f}^k \frac{(\nabla_1 h_{kl})^2}{\k_l-\frac{u}{\l'}}  \nonumber\\
	&+2\sum_{l>1}\dot{f}^1 \frac{\k_l-\frac{u}{\l'}}{\l'^2}|\nabla_l\l'|^2 -2\sum_{l>1}\dot{f}^{l}\frac{\k_l-\frac{u}{\l'}}{\l'^2}|\nabla_l\l'|^2  \nonumber\\
	\geq &-2\sum_{k>1}\frac{\dot{f}^k}{\k_k}(\nabla_1 h_{kk})^2-2\sum_{k\neq l>1}\frac{\dot{f}^k}{\k_l}(\nabla_1 h_{kl})^2+2\sum_{k>1,l>1}\dot{f}^k \frac{(\nabla_1 h_{kl})^2}{\k_l-\frac{u}{\l'}}  \nonumber\\
	&+2\sum_{l>1}\dot{f}^1 \frac{\k_l-\frac{u}{\l'}}{\l'^2}|\nabla_l\l'|^2 -2\sum_{l>1}\dot{f}^{l}\frac{\k_l-\frac{u}{\l'}}{\l'^2}|\nabla_l\l'|^2 \nonumber\\
	=   &2\sum_{k>1,l>1}\dot{f}^k\(\frac{1}{\k_l-\frac{u}{\l'}}-\frac{1}{\k_l}\)(\nabla_1 h_{kl})^2+2\sum_{l>1}(\dot{f}^1-\dot{f}^{l})\frac{\k_l-\frac{u}{\l'}}{\l'^2}|\nabla_l\l'|^2 \nonumber\\
	\geq & 0,
	\end{align}
	where we used \eqref{f-inverse-concave-II} in the second inequality, and the last inequality follows from 
	$$
	0<\frac{u}{\l'}=\k_1 <\k_2< \cdots < \k_n,
	$$
	and $0<\dot{f}^n \leq \cdots \leq \dot{f}^1$ by \eqref{f-concave-II} since $F$ is concave. By the tensor maximum principle, the {\em static convexity} is preserved along the flow \eqref{s2:general-form-flow-I} or the flow \eqref{s2:general-form-flow-II}.
	
	We show that $M_t$ is strictly {\em static convex} for $t>0$. By the strong maximum principle, the inequality becomes strict (and hence the strict increase of the theorem is proved) unless there exists a parallel vector field $v=e_1$ such that $S_{ij}v^iv^j=S_{11}=0$ on $M_{t_0}$. Hence, the smallest principal curvature $\k_1$ is $\frac{u}{\l'}$ on $M_{t_0}$ everywhere, which is strictly less than one since $u = \l\langle \partial_r,\nu\rangle \leq \l <\l'$. This contradicts with the fact that on any closed hypersurface in $\H^{n+1}$, there exists at least one point where all the principal curvatures are strictly greater than one. This completes the proof.
\end{proof}

\begin{rem}
	The above argument can be used to show that the static convexity is also preserved along the (purely) inverse curvature flows 
	$$
	\frac{\partial}{\partial t}X(x,t)=\frac{1}{F^p}\nu(x,t), \quad 0<p\leq 1
	$$ 
	in hyperbolic space $\mathbb H^{n+1}$ provided that $F$ satisfies Assumption \ref{s1:Assumption}.
\end{rem}

\section{Longtime existence and exponential convergence}\label{sec:7} 
In this section, we complete the proof of Theorem \ref{s1:main-thm-I}.

We assume that the initial hypersurface $M_0$ is star-shaped, then there exists a point $o$ inside the enclosed domain $\Omega_0$ by $M_0$. We take $o$ as the origin of the hyperbolic space $\mathbb H^{n+1}$, and express the hyperbolic space $\mathbb H^{n+1}$ as a warped product $\R^+\times \mathbb S^n$ with the metric $\-g=dr^2+\l^2(r)\s$, where $r$ is the distance to the origin $o$. As discussed in $\S \ref{s2:sec-2.2}$, we equivalently write the flow equation \eqref{s1:locally-MCF} as a scalar parabolic PDE on $\mathbb S^n$ for the function $\varphi(\cdot,t)$. By \eqref{s2:evolution-parametric-form}, the function $\varphi$ satisfies
\begin{align}\label{s5:evol-varphi-II}
\frac{\partial}{\partial t}\varphi=\frac{v}{\l}-\frac{E_1}{\l'},
\end{align}
where $v=\sqrt{1+|D\varphi|^2}$. By the expression \eqref{s2:2nd-fundamental-form} of the Weingarten matrix $\mathcal{W}=(h_i^j)$, we have
\begin{align*}
E_1=&\frac{\l'}{\l v}-\frac{1}{n\l v}\(\s^{ki}-\frac{\varphi^k\varphi^i}{v^2}\)\varphi_{ik}.
\end{align*}
Then Eq. \eqref{s5:evol-varphi-II} can be rewritten as a scalar parabolic PDE in divergent form
\begin{align}\label{s5:evol-varphi-II-divergent-type}
\frac{\partial}{\partial t}\varphi=&\frac{1}{n\l\l'v}\(\s^{ki}-\frac{\varphi^k\varphi^i}{v^2}\)\varphi_{ik}+\frac{|D\varphi|^2}{\l v} \nonumber\\
=&\divv\(\frac{D\varphi}{n\l\l'v}\)+\((n+1)\l'^2+\l^2\)\frac{|D\varphi|^2}{n\l \l'^2 v},
\end{align}
where we used $v_k=\frac{\varphi^i\varphi_{ik}}{v}$ and
\begin{align}\label{s5:derivative-lambda}
D_i \l'=\l''r_i=\l^2 \varphi_i, \quad  D_i\l=\l'r_i=\l\l'\varphi_i.
\end{align}

The $C^0$-estimate of the solution to Eq. \eqref{s5:evol-varphi-II} follows from Proposition \ref{s3:prop-C0-estimate}. To deduce the $C^1$-estimate of $\varphi$, we first calculate the evolution equation of $|D\varphi|^2$. 
\begin{lem}\label{s5:lem-evol-Dvarphi2} Let $\varphi$ be a solution of \eqref{s5:evol-varphi-II} and assume that $|D \varphi|^2$ attains its maximum at a point $P$. Then at $P$, there holds
	\begin{align}\label{s5:evol-Dvarphi2-II}
	\frac{\partial}{\partial t}|D\varphi|^2 =&\frac{1}{n\l\l'v}\(\s^{ij}-\frac{\varphi^i\varphi^j}{v^2}\)(|D\varphi|^2)_{ij}-\frac{2}{n\l\l'v}\varphi_i^k\varphi_k^i-\frac{2(n-1)}{n\l\l'v}|D\varphi|^2 \nonumber\\
	&-\frac{2\D \varphi |D\varphi|^2}{n\l v}-\frac{2\D \varphi \l |D\varphi|^2}{n\l'^2 v}-\frac{2\l'|D\varphi|^4}{\l v}.
	\end{align}
\end{lem}
\begin{proof}
	At the maximum point $P$ of $|D\varphi|^2$, we have
	\begin{align*}
	D|D\varphi|^2=0, \quad Dv=0.
	\end{align*} 
	In view of Eq. \eqref{s5:evol-varphi-II-divergent-type}, we derive
	\begin{align}\label{s5:evol-step-1}
	\frac{\partial}{\partial t}|D\varphi|^2=&2\varphi^i D_i \( \frac{\D\varphi}{n\l\l'v}-\frac{v_j\varphi^j}{n\l\l' v^2}+\frac{|D\varphi|^2}{\l v}\) \nonumber\\
	=& 2 \( \frac{ \varphi^i D_i(\D\varphi)}{n\l\l'v}-\frac{\D \varphi |D\varphi|^2}{n\l v}-\frac{\D \varphi \l |D\varphi|^2}{n\l'^2 v}-\frac{v_{ji}\varphi^j\varphi^i}{n\l\l'v^2}-\frac{\l'|D\varphi|^4}{\l v}\).
	\end{align}
	where we used \eqref{s5:derivative-lambda}. Using the Ricci identity on $\mathbb S^n$, we have
	\begin{align*}
	(|D\varphi|^2)_i=&2\varphi^k \varphi_{ki}, \nonumber\\
	(|D\varphi|^2)_{ij}=&2 \varphi_j^k\varphi_{ki}+2\varphi^k \varphi_{kij}\nonumber\\
	=&2 \varphi_j^k\varphi_{ki}+2\varphi^k \(\varphi_{ijk}-\varphi_j \d_k^i+\varphi_k \d_i^j\)\nonumber\\
	=&2 \varphi_j^k\varphi_{ki}+2\varphi^k \varphi_{ijk}-2\varphi_i\varphi_j +2|D\varphi|^2 \d_i^j.
	\end{align*}  
	Then we get
	\begin{align}\label{s5:Ricci-app}
	\varphi^i D_i(\D\varphi)=\frac{1}{2}\D |D\varphi|^2-\varphi_i^k\varphi_k^i-(n-1)|D\varphi|^2.
	\end{align}
	Substituting \eqref{s5:Ricci-app} into \eqref{s5:evol-step-1}, we get
	\begin{align*}
	\frac{\partial}{\partial t}|D\varphi|^2=&\frac{1}{n\l\l'v}\(\D |D\varphi|^2-2\varphi_i^k\varphi_k^i-2(n-1)|D\varphi|^2\)\\
	&-\frac{2\D \varphi |D\varphi|^2}{n\l v}-\frac{2\D \varphi \l |D\varphi|^2}{n\l'^2 v}-\frac{2v_{ji}\varphi^j\varphi^i}{n\l\l'v^2}-\frac{2\l'|D\varphi|^4}{\l v}\\
	=&\frac{1}{n\l\l'v}\(\s^{ij}-\frac{\varphi^i\varphi^j}{v^2}\)(|D\varphi|^2)_{ij}-\frac{2}{n\l\l'v}\varphi_i^k\varphi_k^i-\frac{2(n-1)}{n\l\l'v}|D\varphi|^2 \\
	&-\frac{2\D \varphi |D\varphi|^2}{n\l v}-\frac{2\D \varphi \l |D\varphi|^2}{n\l'^2 v}-\frac{2\l'|D\varphi|^4}{\l v},
	\end{align*}
	where we used $(|D\varphi|^2)_{ij}=(v^2)_{ij}=2v v_{ij}$ since $v_i=0$ at $P$.
\end{proof}

Now we prove the $C^1$-estimate of $\varphi$ and exponential decay of $|D\varphi|^2$. 
\begin{prop}\label{s5:prop-locally-MCF-gradient-estimate}
	Let $\varphi=\varphi(\t,t),t\in[0,T^\ast)$ be a solution of \eqref{s5:evol-varphi-II}. Then there exists a uniform constant $\a>0$ which depends only on the initial hypersurface $M_0$ such that 
	\begin{align}\label{s5:exponential-convergence-II}
	|D\varphi|^2(\t,t) \leq \max_{\t\in \mathbb S^n}|D\varphi|^2(\t,0) e^{-\a t}, \quad t\in [0,T^\ast).
	\end{align}
\end{prop}
\begin{proof}
	By Lemma \ref{s5:lem-evol-Dvarphi2}, at the maximum point $P$ of $|D\varphi|^2$, we have
	\begin{align*}
	\frac{\partial}{\partial t}|D\varphi|^2 \leq &-\frac{2}{n^2\l\l'v}|\D \varphi|^2-\frac{2(n-1)}{n\l\l'v}|D\varphi|^2 -\frac{2\D \varphi |D\varphi|^2}{n\l v}-\frac{2\D \varphi \l |D\varphi|^2}{n\l'^2 v}-\frac{2\l'|D\varphi|^4}{\l v},
	\end{align*}
	where we used the trace inequality $\varphi_i^k\varphi_k^i \geq \frac{1}{n}|\D \varphi|^2$. Then by completing the square, we obtain
	\begin{align*}
	\frac{\partial}{\partial t}|D\varphi|^2 \leq &-\frac{2}{n\l\l'^2v}\( \frac{\l'|\D \varphi|^2}{n}+(\l'^2+\l^2)\D \varphi |D\varphi|^2+\frac{n(\l'^2+\l^2)^2}{4\l'}|D\varphi|^4\)\nonumber\\
	&-\frac{2(n-1)}{n\l\l'v}|D\varphi|^2 +\(\frac{(\l'^2+\l^2)^2}{2\l\l'^3v}-\frac{2\l'}{\l v}\)|D\varphi|^4 \\
	\leq &-\frac{2(n-1)}{n\l\l'v}|D\varphi|^2,
	\end{align*}
	where we used $\frac{(\l'^2+\l^2)^2}{2\l\l'^3}\leq \frac{2\l'}{\l}$ since $\l'>\l$. Therefore, standard maximum principle yields 
	$$
	|D\varphi|^2 (\t,t)\leq C:=\max_{\t\in\mathbb S^n}|D\varphi|^2(\t,0). 
	$$
	Next, we prove the exponential estimate. Since $v=\sqrt{1+|D\varphi|^2}\leq C$, together with the $C^0$ estimate by Proposition \ref{s3:prop-C0-estimate}, there exists a uniform constant $\a\geq \frac{2(n-1)}{n\l\l'v}>0$ such that 
	\begin{align*}
	\frac{\partial}{\partial t}|D\varphi|^2 \leq -\a |D\varphi|^2.
	\end{align*} 
	and hence \eqref{s5:exponential-convergence-II} follows.
\end{proof}

Since Eq. \eqref{s5:evol-varphi-II} can be rewritten as Eq. \eqref{s5:evol-varphi-II-divergent-type}, then by the classical theory of parabolic PDE in divergent form \cite{Lady68}, the higher regularity a priori estimates of the solution $\varphi$ follow from the uniform gradient estimate in Proposition \ref{s5:prop-locally-MCF-gradient-estimate}. Moreover, the solution exists for all $t\in [0,\infty)$, i.e., $T^\ast=\infty$. The exponential convergence to a geodesic sphere $\partial B_{r_\infty}$ centered at the origin follows from the estimate \eqref{s5:exponential-convergence-II}. Since $W_0^{\l'}(\Omega_t)$ is constant along the flow \eqref{s1:locally-MCF} with $F=E_1$, the radius $r_\infty$ is uniquely determined by $W_0^{\l'}(B_{r_\infty})=W_0^{\l'}(\Omega_0)$, where $\Omega_0$ is the bounded domain enclosed by the initial hypersurface $M_0$. By Theorem \ref{thm-static-convexity}, if the initial hypersurface $M_0$ is static convex, then the flow hypersurface $M_t$ becomes strictly static convex for $t>0$. This completes the proof of Theorem \ref{s1:main-thm-I}.

\section{Applications to geometric inequalities}\label{sec:8} 
In this section, we give the proofs of Theorems \ref{thm-weighted-quermassintegral-ineq-I}, \ref{thm-geometric-ineq} and \ref{thm-geometric-ineq-I}.

\begin{proof}[Proof of Theorem \ref{thm-weighted-quermassintegral-ineq-I}]
	{\bf Case 1.} ($1\leq k\leq n-1$.) We choose $F=E_1$ in the flow \eqref{s1:locally-MCF}. By Theorem \ref{s1:main-thm-I}, the flow hypersurface $M_t$ is static convex and hence it is strictly convex. By \eqref{s6:evol-weighted-curvature-integral}, we have
	\begin{align*}
	\frac{d}{dt}W^{\l'}_0(\Omega_t)=&(n+1)\int_{M_t}\l'\(1-\frac{uE_1}{\l'}\)d\mu_t=0,
	\end{align*}
	and for $1\leq k\leq n-1$,
	\begin{align*}
	\frac{d}{dt}W^{\l'}_{k}(\Omega_t)=&\int_{M_t}( kuE_{k-1}+(n+1-k)\l'E_k)\(1-\frac{uE_1}{\l'}\)d\mu_t \\
	                               =&k\int_{M_t} \frac{u}{\l'}(\l'E_{k-1}-u E_{k-1}E_1)d\mu_t+(n+1-k)\int_{M_t} (\l'E_k-uE_kE_1)d\mu_t \\
	                               \leq &k \int_{M_t} \frac{u}{\l'}(\l'E_{k-1}-u E_k) d\mu_t,
	\end{align*}
	where we used Newton-MacLaurin inequality $E_{k-1} E_1 \geq E_k$ and $E_k E_1 \geq E_{k+1}$ for $1\leq k\leq n-1$ and Minkowski formula \eqref{s2:Minkowski-formula}. 
	
	It follows from \eqref{newton-formula-1}, \eqref{newton-formula-2} and \eqref{s2:2.1} that
	\begin{align}\label{s6:Minkowski}
	\dot{E}_m^{ij} \nabla_i \nabla_j \l'=\dot{E}_m^{ij}(\l'g_{ij}-uh_{ij})=m(\l'E_{m-1}-uE_m), \quad m=1,\cdots,n.
	\end{align} 
	Since $\dot{E}_k^{ij}$ is divergence-free, by \eqref{s6:Minkowski} and integration by parts, we deduce that
	\begin{align}\label{s6:monotone}
	\frac{d}{dt}W^{\l'}_{k}(\Omega_t) \leq& \int_{M_t} \frac{u}{\l'}  \dot{E}_{k}^{ij} \nabla_i\nabla_j\l' d\mu_t \nonumber\\
	= &-\int_{M_t}  \dot{E}_{k}^{ij}\nabla_i(\frac{u}{\l'})\nabla_j\l' d\mu_t \nonumber \\
	= &-\int_{M_t}  \dot{E}_{k}^{ii}\frac{\l'\k_i-u}{\l'^2}|\nabla_i\l'|^2 d\mu_t \leq 0.
	\end{align}
	Here the last inequality follows from the fact that $\l'\k_i-u\geq 0$ and $\dot{E}_k^{ij}$ is positive definite, since hypersurface $M_t$ is static convex. By Theorem \ref{s1:main-thm-I}, we get
	\begin{align*}
	W^{\l'}_{k}(\Omega_0) \geq &W^{\l'}_{k}(B_{r_{\infty}})= h_{k}(r_{\infty}) = h_{k}\circ h_{0}^{-1}(W^{\l'}_{0}(B_{r_\infty})) \nonumber\\
	=&h_{k}\circ h_{0}^{-1}(W^{\l'}_{0}(\Omega_0)), \quad 1\leq k\leq n-1.
	\end{align*}
	
	{\bf Case 2.} ($2\leq k \leq n+1$.) We choose $F=E_{k-1}/E_{k-2}$ in the flow \eqref{s1:SX-ICF}. By Theorem \ref{s1:main-thm-III}, the flow hypersurface $M_t$ is static convex and hence it is strictly convex. In view of \eqref{s6:evol-weighted-curvature-integral}, we have
	\begin{align}\label{s6:case-2-monotonicity-1}
	\frac{d}{dt}W^{\l'}_{0}(\Omega_t)=&(n+1)\int_{M_t}\(\l'\frac{E_{k-2}}{E_{k-1}}-u\) d\mu_t \geq (n+1)\int_{M_t}\(\frac{\l'}{E_1}-u\) d\mu_t \geq 0,
	\end{align}
	where we used Newton-MacLaurin inequality $E_{1}E_{k-2} \geq E_{k-1}$ for $2\leq k \leq n+1$ in the first inequality, and Heintze-Karcher inequality \cite[Theorem 3.5]{Brendle2013} in the second inequality. 
	
	On the other hand, we derive
	\begin{align}\label{s6:case-2-monotonicity-2}
	\frac{d}{dt}W^{\l'}_{k}(\Omega_t)=&\int_{M_t}(k u E_{k-1}+(n+1-k)\l'E_{k})\(\frac{E_{k-2}}{E_{k-1}}-\frac{u}{\l'}\) d\mu_t \nonumber\\
	                                    =&k \int_{M_t} \frac{u}{\l'} \( \l'E_{k-2}-u E_{k-1}\) d\mu_t+(n+1-k)\int_{M_t}\(\l'\frac{E_{k}E_{k-2}}{E_{k-1}}-uE_{k}\)d\mu_t \nonumber\\
	                                    \leq &k \int_{M_t}\frac{u}{\l'} (\l' E_{k-2}-uE_{k-1})d\mu_t,
	\end{align}
	where we used Newton-MacLaurin inequality $E_{k}E_{k-2} \leq E_{k-1}^2$ for $2\leq k \leq n+1$ and Minkowski formula \eqref{s2:Minkowski-formula}. Since $\dot{E}_{k-1}^{ij}$ is divergence-free, by integration by parts and \eqref{s6:Minkowski}, we have
	\begin{align}\label{s6:monotone-locally-MCF}
	\frac{d}{dt}W^{\l'}_{k}(\Omega_t) \leq& \int_{M_t} \frac{k}{k-1}\frac{u}{\l'}  \dot{E}_{k-1}^{ij}\nabla_i\nabla_j\l' d\mu_t \nonumber\\
	= &-\frac{k}{k-1}\int_{M_t}  \dot{E}_{k-1}^{ij}\nabla_i(\frac{u}{\l'})\nabla_j\l' d\mu_t \nonumber \\
	= &-\frac{k}{k-1}\int_{M_t}  \dot{E}_{k-1}^{ii}\frac{\l'\k_i-u}{\l'^2}|\nabla_i\l'|^2 d\mu_t \leq 0,
	\end{align}
	where the last inequality follows from $\l'\k_i -u\geq 0$ and $\dot{E}_{k-1}^{ij}$ is positive definite, since $M_t$ is static convex (by Theorem \ref{s1:main-thm-III}). Combining \eqref{s6:case-2-monotonicity-1} and \eqref{s6:case-2-monotonicity-2} with the convergence of the flow \eqref{s1:SX-ICF} (Theorem B), we deduce that
	\begin{align*}
	W^{\l'}_{k}(\Omega_0) \geq &W^{\l'}_{k}(B_{r_{\infty}})= h_{k}(r_{\infty}) = h_{k}\circ h_{0}^{-1}(W^{\l'}_{0}(B_{r_\infty})) \\
	\geq &h_{k}\circ h_{0}^{-1}(W^{\l'}_{0}(\Omega_0)), \quad 2\leq k\leq n+1.
	\end{align*}
	If equality holds in \eqref{weighted-quermassintegral-ineq-I}, then equality also holds in \eqref{s6:monotone} or \eqref{s6:monotone-locally-MCF}. Since $M_t$ is strictly static convex for $t>0$, this implies that $\nabla\l'\equiv 0$ on $M_t$ and hence $M_t$ must be a geodesic sphere centered at the origin for $t>0$. Therefore, by smooth approximation of $M_t$ as $t\ra 0$, the initial hypersurface $M_0$ is also a geodesic sphere centered at the origin. This completes the proof of Theorem \ref{thm-weighted-quermassintegral-ineq-I}.
\end{proof}

\begin{proof}[Proof of Theorem \ref{thm-geometric-ineq}]
	We choose $F=E_1$ in the flow \eqref{s1:locally-MCF}. In view of \eqref{s6:evol-weighted-curvature-integral}, we have
	\begin{align*}
	\frac{d}{dt}W^{\l'}_{0}(\Omega_t)=&(n+1)\int_{M_t} \l'\(1-\frac{uE_1}{\l'}\)d\mu_t=0,
	\end{align*}
	where we used Minkowski formula \eqref{s2:Minkowski-formula}. On the other hand, we have
	\begin{align}\label{s7:monotonicity-geometric-ineq-I}
	\frac{d}{dt}W_0(\Omega_t)=&(n+1)\int_{M_t} \(1-\frac{uE_1}{\l'}\)d\mu_t \nonumber\\ 
	                         =&\frac{n+1}{n} \int_{M_t} \frac{\D \l'}{\l'} d\mu_t \nonumber\\
	                         =&\frac{n+1}{n} \int_{M_t} \frac{|\nabla \l'|^2}{\l'^2} d\mu_t\geq 0.
	\end{align}
	By Theorem \ref{s1:main-thm-I}, the flow hypersurface $M_t$ converges smoothly to a geodesic sphere $\partial B_{r_\infty}$ centered at the origin, we obtain
	\begin{align*}
	W^{\l'}_{0}(\Omega_0)= &W^{\l'}_{0}(B_{r_{\infty}})=h_0(r_\infty)=h_0\circ f_0^{-1}(W_{0}(B_{r_{\infty}})) \\
			          \geq &h_0\circ f_0^{-1}(W_{0}(\Omega_0)).
	\end{align*} 
	If equality holds in \eqref{geom-ineq-1}, then equality holds in \eqref{s7:monotonicity-geometric-ineq-I}, which implies that the initial hypersurface $M_0$ is a geodesic sphere centered at the origin. This completes the proof of Theorem \ref{thm-geometric-ineq}. 
\end{proof}

\begin{proof}[Proof of Theorem \ref{thm-geometric-ineq-I}]
	
	{\bf Case 1.} ($0\leq k \leq n$ and $m=0$.) This follows immediately from Theorems \ref{thm-weighted-quermassintegral-ineq-I} and \ref{thm-geometric-ineq}, that is,
	$$
	W_{k+1}^{\l'}(\Omega_0)\geq h_{k+1}\circ h_0^{-1}(W_{0}^{\l'}(\Omega_0))\geq h_{k+1}\circ f_0^{-1}(W_{0}(\Omega_0)).
	$$
	Moreover, if equality holds in \eqref{geom-ineq-2} in Case 1, the initial hypersurface $M_0$ is a geodesic sphere centered at the origin. 
	
	{\bf Case 2.} ($1\leq k\leq n$ and $1\leq m\leq k$.) We choose $F=E_{k}/E_{k-1}$ in the flow \eqref{s1:SX-ICF}. By Theorem \ref{s1:main-thm-III}, the flow hypersurface $M_t$ is static convex and hence it is strictly convex.  
	
	By \eqref{s2:variation-quermassintegral} and \eqref{s6:Minkowski}, we derive
	\begin{align}\label{s6:geom-ineq-2}
	\frac{d}{dt}W_m(\Omega_t)=&\frac{n+1-m}{n+1}\int_{M_t} E_m \(\frac{E_{k-1}}{E_{k}}-\frac{u}{\l'}\)d\mu_t  \nonumber\\
	\geq &\frac{n+1-m}{n+1}\int_{M_t}\frac{\l'E_{m-1}-uE_m}{\l'}d\mu_t \nonumber\\
	=&\frac{n+1-m}{m(n+1)}\int_{M_t} \frac{\dot{E}_{m}^{ij}\nabla_i\l'\nabla_j\l'}{\l'^2}d\mu_t\geq 0,
	\end{align}
	where in the first inequality we used $E_m E_{k-1} \geq E_{m-1}E_{k}$ for $1\leq m\leq k$, and the last inequality follows from the positivity of $\dot{E}_m^{ij}$.
	
	On the other hand, by \eqref{s6:evol-weighted-curvature-integral} we have
	\begin{align*}
	\frac{d}{dt}W^{\l'}_{k+1}(\Omega_t)  = & \int_{M_t} \( (k+1)uE_{k}+(n-k)\l'E_{k+1}\)\(\frac{E_{k-1}}{E_{k}}-\frac{u}{\l'}\)d\mu_t\\
	= &(k+1) \int_{M_t} \frac{u}{\l'}(\l' E_{k-1}-u E_{k})d\mu_t+(n-k)\int_{M_t} \(\l'\frac{E_{k+1}E_{k-1}}{E_{k}}-uE_{k+1}\)d\mu_t \\
	\leq & (k+1) \int_{M_t} \frac{u}{\l'} (\l'E_{k-1}-uE_{k})d\mu_t,
	\end{align*}
	where we used Newton-MacLaurin inequality $E_{k+1}E_{k-1}\leq E_{k}^{2}$, and Minkowski formula \eqref{s2:Minkowski-formula} and $E_{n+1}=0$ by convention. Finally, by using \eqref{s6:Minkowski} we get
	\begin{align}\label{s6:Monotonicity-formula}
	\frac{d}{dt}W^{\l'}_{k+1}(\Omega_t) \leq  & -\frac{k+1}{k}\int_{M_t}\dot{E}_{k}^{ij}\nabla_i(\frac{u}{\l'})\nabla_j\l' d\mu_t\nonumber\\
	=&-\frac{k+1}{k} \int_{M_t}\sum_{i}\dot{E}_{k}^{ii}\frac{\l'\k_i-u}{\l'^2}|\nabla_i\l'|^2 d\mu_t \leq 0,
	\end{align}
	where the last inequality follows from $\l'\k_i -u \geq 0$ and $\dot{E}_{k}^{ij}$ is positive definite, since $M_t$ is static convex.
	
	By \eqref{s6:geom-ineq-2}, \eqref{s6:Monotonicity-formula} and the convergence of the flow \eqref{s1:SX-ICF} with $F=E_{k}/E_{k-1}$ (Theorem B), we have
	\begin{align*}
	W^{\l'}_{k+1}(\Omega_0) \geq &W^{\l'}_{k+1}(B_{r_\infty}) = h_{k+1}(r_\infty) =h_{k+1}\circ f_m^{-1}(W_m(B_{r_\infty})) \\
	                        \geq &h_{k+1}\circ f_m^{-1}(W_m(\Omega_0)), \quad 1\leq m\leq k.
	\end{align*}
	If equality holds in \eqref{geom-ineq-2} in Case 2, then equality holds in \eqref{s6:Monotonicity-formula}. Since $M_t$ is strictly static convex for $t>0$, the similar argument as above implies that the initial hypersurface $M_0$ is a geodesic sphere centered at the origin. This completes the proof of Theorem \ref{thm-geometric-ineq-I}.   
\end{proof}

\end{document}